\documentclass[11pt]{amsart}
\usepackage{a4wide}
\usepackage{amsmath} \usepackage{amsthm} \usepackage{amscd}
\usepackage{amssymb} \usepackage{graphicx} 
\usepackage[labelformat=empty]{subfig}

\newtheorem{theorem}{Theorem}[section]
\newtheorem{lemma}[theorem]{Lemma}
\newtheorem{corollary}[theorem]{Corollary}
\newtheorem{proposition}[theorem]{Proposition}

\theoremstyle{definition} \newtheorem{definition}[theorem]{Definition}
\newtheorem{example}[theorem]{Example}

\theoremstyle{remark} \newtheorem*{remark}{Remark}
\newtheorem*{nota}{Notation}

\numberwithin{equation} {section} \newcounter{temp}

\makeatletter \def\square{\RIfM@\bgroup\else$\bgroup\aftergroup$\fi
\vcenter{\hrule\hbox{\vrule\@height.6em\kern.6em\vrule}\hrule}\egroup}
\makeatother

\DeclareMathOperator{\Spec}{Spec} \DeclareMathOperator{\Proj}{Proj}
\DeclareMathOperator{\Hom}{Hom} 
\DeclareMathOperator{\image}{im} 
\DeclareMathOperator{\Coker}{Coker} 
 \DeclareMathOperator{\link}{lk}

 \DeclareMathOperator{\Def}{Def}
\DeclareMathOperator{\rank}{rank} 
 \DeclareMathOperator{\vertices}{vert}
\DeclareMathOperator{\Sat}{Sat} \DeclareMathOperator{\lcm}{lcm}
 \DeclareMathOperator{\GL}{GL}
\DeclareMathOperator{\PGL}{PGL} \DeclareMathOperator{\Aut}{Aut}
\DeclareMathOperator{\diag}{diag} \DeclareMathOperator{\interior}{int}
\DeclareMathOperator{\SL}{SL} 

\begin{document}

\title{Deformations of equivelar Stanley-Reisner abelian surfaces}
\author{Jan Arthur Christophersen} \address{Matematisk institutt,
Postboks 1053 Blindern, University of Oslo, N-0316 Oslo, Norway}
\email{christop@math.uio.no} \thanks{ I am grateful to Klaus Altmann,
Kristian Ranestad, Benjamin Nill, Maxmillian Kreuzer, Christian Hasse
and Frank Lutz for helpful discussions and answering questions. Much
of this work was done during a years visit at the Johannes
Gutenberg-Universit\"{a}t Mainz.  The stay was funded by the
Sonderforschungsbereich/Transregio 45: Periods, Moduli Spaces and
Arithmetic of Algebraic Varieties.}  \date{\today}
\begin{abstract} The versal deformation of Stanley-Reisner schemes
associated to equivelar triangulations of the torus is studied. The
deformation space is defined by binomials and there is a toric
smoothing component which I describe in terms of cones and lattices.
Connections to moduli of abelian surfaces are considered. The case of
the M\"{o}bius torus is especially nice and leads to a projective
Calabi-Yau 3-fold with Euler number 6.
\end{abstract}

\maketitle

 \section*{Introduction} Versal deformation spaces in algebraic
geometry tend to be either smooth, i.e. the object to deform is
unobstructed, or much too complicated to compute. In general the
equations defining formal versal base spaces are not polynomials. The
purpose of this paper is to present an exception. In \cite{ac:def} we
showed that triangulated surface manifolds with regular edge graph of
degree $6$, give Stanley-Reisner schemes with nicely presented formal
versal deformations spaces defined by polynomials, in fact
binomials. Such a surface is either a torus or a Klein bottle.

In this paper, the torus case is studied. Triangulations of (or more
generally maps on) surfaces with regular edge graph are called
equivelar.  In the torus case they were only recently classified and
counted by Brehm and K\"{u}hnel in \cite{bk:equ}. The Stanley-Reisner
schemes of such tori are all smoothable and they smooth to abelian
surfaces.

To avoid non-algebraic abelian surfaces I will work with the functor
$\Def_{(X,L)}$ where $X$ is a scheme and $L$ is an invertible
sheaf. Since a projective Stanley-Reisner scheme $X$ comes equipped
with a very ample line bundle I define $\Def^{a}_{X} = \Def_{(X,
\mathcal{O}_{X}(1))}$ and it is the versal formal element of this
functor I consider.

Let $T$ be an equivelar triangulated torus and $X$ the Stanley-Reisner
scheme. Since the equations defining $\Def^{a}_{X}$ are binomial, we
may use the results and techniques of Eisenbud and Sturmfels in
\cite{es:bin} to realize the smoothing components as toric
varieties. The main computational part, Section \ref{analysis}, is
about the cones and lattices that determine these toric varieties (or
their normalizations). I then apply these results to statements about
the deformations of $X$ in Section \ref{def}.

There should be a connection between the results in this paper and
moduli of polarized abelian surfaces. In Section \ref{modulisec} I
describe a Heisenberg group $H_{T}$ associated to $T$. There turns out
to be a smooth 3 dimensional subspace $\mathcal{M} \subset
\Def_{X}^{a}$ containing all isomorphism classes of smoothings of
$X$. Moreover, the fibers are exactly the $H_{T}$ invariant
deformations of $X$. There is a finite group acting on $\mathcal{M}$
inducing isomorphisms on the fibers and the quotient space
$\bar{\mathcal{M}}$ can be easily described in toric geometric terms.

To understand the connection to moduli one should extend the results
in this paper to the non-polyhedral equivelar maps on the torus. These
should correspond to moduli where the polarization class is not
represented by a very ample line bundle. This is at the moment work in
progress.

In principle one can write equations for abelian surfaces in
$\mathbb{P}^{n-1}$ as perturbations of the Stanley Reisner ideal of
$T$. I include some details about this ideal in
Section~\ref{SR-scheme}. An application can be found in
Section~\ref{T7}.

The last section deals in detail with the vertex minimal triangulation
of the torus, sometimes called the M\"{o}bius torus, with 7
vertices. Here the toric geometry of $\Def^{a}_{X}$ is extremely nice
and leads to a Calabi-Yau 3-fold with Euler number 6. In this case it
is also possible to find all the components of $\Def_{X}^{a}$ and the
generic non-smoothable fibers.

It is convenient to work over the ground field $\mathbb{C}$. I use the
notation $\Def^{a}_{X}$ for \emph{both} the functor and the versal
base space. Throughout this paper $G^{\ast}: = \Hom_{\mathbb{Z}}(G,
\mathbb{C}^{\ast})$ is the character group of $G$.

  
  \section{Preliminaries}
 
\subsection{Equivelar triangulations of the torus} \label{about36} I
start by defining the main combinatorial object in this paper. For
details and proofs see \cite{bk:equ}. A map on a surface is called
\emph{equivelar} if there are numbers $p$ and $q$ such that every
vertex is $q$-valent and every facet contains exactly $p$ vertices. On
a torus we can only have $(p,q)$ equal $(6,3)$, $(3,6)$ or $(4,4)$.

We will consider \emph{triangulated} tori, i.e. the case $p=3,
q=6$. In this paper we need honest triangulations and assume the map
is \emph{polyhedral}. This means that the intersection of two
triangles is a common face (i.e. empty, vertex or edge).

Every triangulated torus (also the non-polyhedral) are obtained as a
quotient of the regular tessellation of the plane by equilateral
triangles (\cite{neg:uni}).  Denote this tessellation $\{3,6\}$. We
will need to make this explicit and will refer to the following as the
\emph{standard description}.

We may describe $\{3,6\}$ as an explicit triangulation of
$\mathbb{R}^{2}$, i.e. we always assume a chosen origin $0$ and
coordinates $(x,y)$. We may assume the vertices of $\{3,6\}$ form the
rank $2$ lattice spanned by $(1,0)$ and
$\frac{1}{2}(1,\sqrt{3})$. Denote this lattice by $\mathbb{T}$.  We
may think of $\mathbb{T}$ as the translation subgroup of
$\Aut(\{3,6\})$. Now let $\Gamma \subseteq \mathbb{T}$ be a sublattice
of finite index and set $T = \{3,6\}/\Gamma$. Then $T$ is a (not
necessarily polyhedral) equivelar triangulated torus.  Such a
triangulation is called \emph{chiral} if the $\mathbb{Z}_{6}$ rotation
on $\{3,6\}$ descends to $T$. Chiral maps on the torus were studied
and classified in \cite{co:con}.

We may assume that $\Gamma$ is generated by $a(1,0)$ and $b(1,0) +
c\frac{1}{2}(1,\sqrt{3})$ for integers $a,b,c$ with $ac \ne 0$,
i.e. it is the image of
 $$
\begin{pmatrix} a & b\\ 0 & c
\end{pmatrix} $$ in the above basis for $\mathbb{T}$. In
\cite[Proposition 2]{bk:equ} it is shown that two such matrices,
$M_{1}$ and $M_{2}$, represent isomorphic triangulated tori if and
only if $M_{2} = P M_{1} Q$, with $Q \in \GL_{2}(\mathbb{Z})$ and $P$
in the $D_{6}$ subgroup generated by the rotation $\rho$ and
reflection $\sigma$,
$$
\rho =
\begin{pmatrix} 0 & -1\\ 1 & 1
\end{pmatrix} \quad \sigma =
\begin{pmatrix} -1 & -1\\ 0 & 1
\end{pmatrix} \, .$$

\begin{figure} \centering
\includegraphics[scale=0.9]{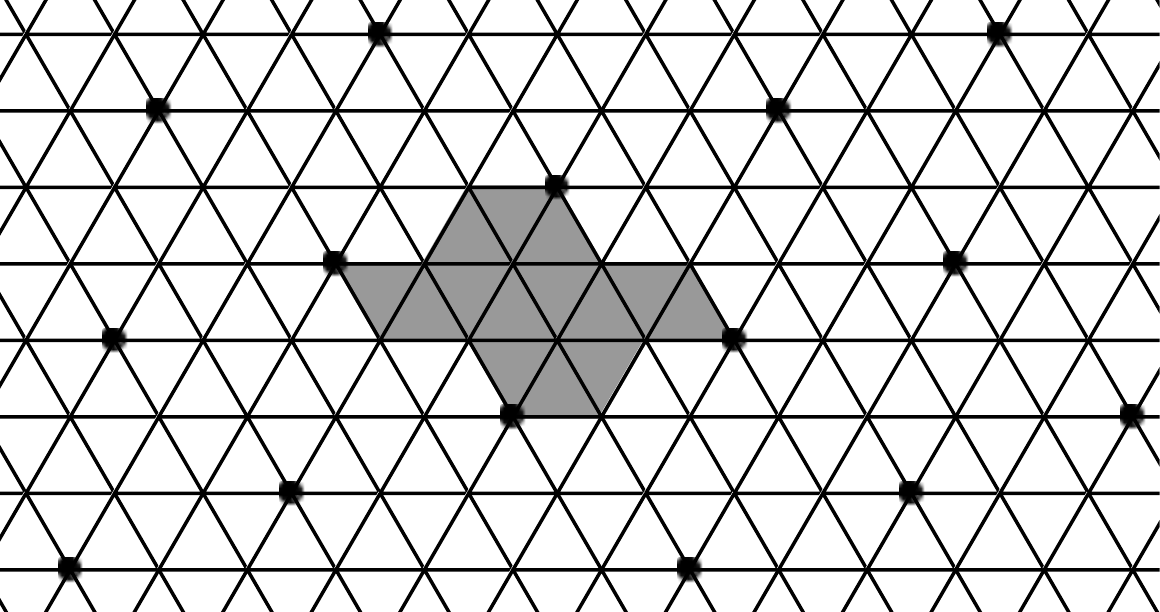}
\caption{The tessellation $\{3,6\}$ with marked lattice points for the
$\Gamma$ with $a=7, b=2,c=1$. The shaded area is a fundamental domain
for the $\Gamma$ action on $\{3,6\}$.}
\label{lattice}
\end{figure}

\subsection{Deformations of Stanley-Reisner schemes} \label{simp} I
refer to \cite{ac:def} and the references there for definitions,
details and proofs about deformations of Stanley-Reisner schemes. As a
general reference for deformation theory see \cite{ser:def}.

Let $[n]$ be the set $\{0,\ldots,n-1\}$ and let
$\Delta_{n-1}:=2^{[n]}$ be the full simplex. Let
$P=\mathbb{C}[x_0,\ldots,x_{n-1}]$ be the polynomial ring in $n$
variables.  If $a=\{i_{1},\dots ,i_{k}\}\in \Delta_{n-1}$, we write
$x_a\in P$ for the square free monomial $x_{i_1}\cdots x_{i_k}$.  A
simplicial complex $\mathcal{K}\subseteq \Delta_{n-1}$ gives rise to
an ideal
\[ I_\mathcal{K}:=\langle x_{p} \, | \, p\in \Delta_{n-1}\setminus
\mathcal{K}\rangle \subseteq P .
\] The {\it Stanley-Reisner ring} is then
$A_\mathcal{K}=P/I_\mathcal{K}$.  We refer to \cite{sta:com} for more
on Stanley-Reisner rings. The corresponding projective Stanley-Reisner
scheme is ${\mathbb P}(\mathcal{K}) = \Proj A_\mathcal{K} \subseteq
\mathbb{P}^{n-1}$. Note that ${\mathbb P}(\mathcal{K})$ comes with a
very ample line bundle $\mathcal{O}_{{\mathbb P}(\mathcal{K})}(1)$.

The scheme ${\mathbb P}(\mathcal{K})$ looks like the geometric
realization of $\mathcal{K}$. It is a union of irreducible components
$X_{F} = {\mathbb P}^{\dim F}$, $F$ a facet of $\mathcal{K}$,
intersecting as in $\mathcal{K}$. There is also a natural open affine
cover described in terms of Stanley-Reisner rings.  Recall that the
\emph{link} of a face is $$ \link(f,\mathcal{K}) = \{g \in \mathcal{K}
: g \cap f = \emptyset \text{ and } g \cup f \in \mathcal{K} \} \, .$$
If $f \in 2^{[n]}$, let $D_{+}(x_f) \subseteq {\mathbb
P}(\mathcal{K})$ be the chart corresponding to homogeneous
localization of $A_\mathcal{K}$ by the powers of $x_f$. Then
$D_{+}(x_f)$ is empty unless $f\in \mathcal{K}$ and if $f\in
\mathcal{K}$ then $$ D_{+}(x_f) = {\mathbb A}(\link(f,\mathcal{K}))
\times (\mathbb{C}^{*})^{\dim f} $$ where ${\mathbb A}(\mathcal{K})$
denotes $\Spec A_\mathcal{K}$.

The cohomology of the structure sheaf is given by $H^p({\mathbb
P}(\mathcal{K}),\mathcal{O}_{{\mathbb P}(\mathcal{K})}) \simeq
H^{p}(\mathcal{K};\mathbb{C})$ (Hochster, see \cite[Theorem
2.2]{ac:def}). If $\mathcal{K}$ is an \emph{orientable} combinatorial
manifold then the canonical sheaf is trivial (\cite[Theorem
6.1]{be:gra}). Thus a smoothing of such a ${\mathbb P}(\mathcal{K})$
would yield smooth schemes with trivial canonical bundle and structure
sheaf cohomology equaling $H^{p}(\mathcal{K};\mathbb{C})$. In
particular if $\mathcal{K}$ comes from a triangulation of a $2$
dimensional torus then a smoothing of ${\mathbb P}(\mathcal{K})$ is an
abelian surface.

In the surface case deformations of Stanley-Reisner schemes may
include non-algebraic schemes. It is therefore convenient to work with
the functor $\Def_{(X,L)}$ where $X$ is a scheme and $L$ is an
invertible sheaf. (See \cite[3.3.3]{ser:def} and \cite[3]{ac:def}.) We
defined $\Def^{a}_{\mathbb{P}(\mathcal{K})} =
\Def_{(\mathbb{P}(\mathcal{K}), \mathcal{O}_{{\mathbb
P}(\mathcal{K})}(1))}$. If $\mathcal{K}$ is a combinatorial manifold
without boundary then $$ \Def_{{\mathbb
P}(\mathcal{K})}^{a}(\mathbb{C}[\epsilon]) \simeq
H^0(\mathbb{P}(\mathcal{K}), {\mathcal T}_{\mathbb{P}(\mathcal{K})}^1)
\simeq T^1_{A_\mathcal{K},0}$$ and $H^0(\mathbb{P}(\mathcal{K}),
{\mathcal T}_{\mathbb{P}(\mathcal{K})}^2)$ contains all obstructions
for $\Def_{{\mathbb P}(\mathcal{K})}^{a}$ (\cite[Theorem
6.1]{ac:def}). For certain surfaces the versal base space for
$\Def^{a}_{\mathbb{P}(\mathcal{K})}$ may be computed and as we shall
see this is particularly nice for equivelar triangulated tori.

It follows from the results in \cite{ac:def} that if $\mathcal{K}$ is
a combinatorial manifold without boundary and all vertices have
valency greater than or equal 5, then $T^1_{A_\mathcal{K},0}$ is the
$\mathbb{C}$ vector space on the edges of $\mathcal{K}$. Since
$\mathcal{K}$ is a manifold, the link of an edge must be two vertices,
i.e. $\link(\{p,q\}) = \{\{i\}, \{j\}\}$. If $\varphi_{p,q} \in
T^1_{A_\mathcal{K},0}$ is the basis element corresponding to $\{p,q\}$
and $x_{m}$ is in the Stanley-Reisner ideal, then
$$
\varphi_{p,q}(x_{m}) =
\begin{cases} \frac{x_{m}x_{p}x_{q}}{x_{i}x_{j}} \quad & \text{if
$\{i,j\} \subseteq m$} \\ 0 & \text{otherwise.}
\end{cases}$$

There is a natural $(n-1)$-dimensional torus action on
$\Proj(A_\mathcal{K})$ where $[\lambda_{0}, \dots , \lambda_{n-1}] \in
(\mathbb{C}^{\ast})^{n}/\mathbb{C}^{\ast}$ takes $x_{i} \in
A_\mathcal{K}$ to $\lambda_{i}x_{i}$. Note that the induced action on
a $\varphi_{p,q} \in T^1_{A_\mathcal{K},0}$ as above is
$$
\varphi_{p,q} \mapsto
\frac{\lambda_{p}\lambda_{q}}{\lambda_{i}\lambda_{j}}\varphi_{p,q}\,
.$$ If $t_{p,q}$ is the corresponding coordinate function on the
versal base space then the action is the contragredient, i.e. $t_{p,q}
\mapsto (\lambda_{i}\lambda_{j}/\lambda_{p}\lambda_{q})t_{p,q} $.

\subsection{Binomial ideals} \label{latticeid} In \cite{es:bin}
Eisenbud and Sturmfels prove, among other things, that every binomial
ideal has a primary decomposition all of whose primary components are
binomial. I review here some of the results and notions from that
paper. (See also \cite{dmm:com}).

If $w = (a_{1}, \dots , a_{n}) \in \mathbb{Z}^{n}_{\ge 0}$, write
$t^{w} = \prod t_{i}^{a_{i}}$ for a monomial in $P = k[t_{1}, \dots ,
t_{n}]$, $k$ for the time being is any algebraically closed field.  A
binomial is a polynomial with at most two terms, $at^{v} - bt^{w}$
with $a,b \in k$. A binomial ideal is an ideal of $P$ generated by
binomials.

For an integer vector $v$, let $v_{+}$ and $v_{-}$, both with
non-negative coordinates, be the positive and negative part of $v$,
i.e. $v = v_{+} - v_{-}$. In general define, for a sublattice $L
\subseteq \mathbb{Z}^{n}$ the \emph{lattice ideal} of $L$ by $$ I_{L}
= \langle t^{v_{+}} - t^{v_{-}} : v \in L \rangle \subseteq k[t_{1},
\dots , t_{n}] \, .$$ More generally for any character $\rho \in
\Hom_{\mathbb{Z}}(L,k^{\ast})$, define
 $$
I_{L,\rho} = \langle t^{v_{+}} -\rho(v) t^{v_{-}} : v \in L \rangle \,
.$$ If $\rho^{\prime}$ is an extension of $\rho$ to $\mathbb{Z}^{n}$,
then the automorphism $t_{i} \mapsto
\rho^{\prime}(\varepsilon_{i})t_{i}$ induces an isomorphism $I_{L}
\simeq I_{L,\rho}$.

Define the \emph{saturation} of $L$ in $\mathbb{Z}^{n}$ as the
lattice $$ \Sat L = \{ v \in \mathbb{Z}^{n} : d v \in L \text{ for
some $d \in \mathbb{Z}$} \} \, .$$ Note $\Sat L/L$ is finite. The
lattice $L$ is \emph{saturated in $\mathbb{Z}^{n}$} if $\Sat L =
L$. The lattice ideal is a prime ideal if and only if $L$ is saturated
(\cite[Theorem 2.1]{es:bin}). In fact \cite[Corollary 2.3]{es:bin}
states that
$$
I_{L} = \bigcap_{\rho \in (\Sat L/L)^{\ast}} I_{\Sat L, \rho}$$ is a
minimal primary decomposition.

Let $I$ be a binomial ideal. Let $\mathcal{Z} \subseteq \{1, \dots
,n\}$ and let $\mathfrak{p}_{\mathcal{Z}} = \langle t_{i} : i \notin
\mathcal{Z}\rangle$ and $\mathbb{Z}^{\mathcal{Z}} \subset
\mathbb{Z}^{n}$ the sublattice spanned by the standard basis elements
$e_{i}, i \in \mathcal{Z}$. The result we need from \cite{es:bin} is
the following. (It is not stated in the following form in that paper
and in fact much stronger results are proven there.)
\begin{theorem} \label{es} If characteristic $k$ is 0, the associated
primes of the binomial ideal $I$ are all of the form $I _{\Sat
L_{\mathcal{Z}}, \rho} + \mathfrak{p}_{\mathcal{Z}}$ for some
sublattice $L_{\mathcal{Z}} \subseteq \mathbb{Z}^{\mathcal{Z}}$.
\end{theorem}

For simplicity let us assume $I$ is generated by pure binomials of the
form $t^{v} - t^{w}$. Define the \emph{exponent vector} of $t^{v} -
t^{w}$ to be $v - w \in \mathbb{Z}^{m}$. Let $L \subset
\mathbb{Z}^{m}$ be the sublattice spanned by the exponent vectors of
the generators of $I$. It follows from the above and \cite[Theorem
6.1]{es:bin} that the $ I_{\Sat L, \rho}$ will be minimal prime ideals
for $I$. It follows from the theorem that all other associated primes
must contain some variable $t_{i}$.

\subsection{Gorenstein, reflexive and Cayley cones} I recall some
notions originally introduced in \cite{bb:dua} in connection with
mirror symmetry.  A general reference is \cite{bn:com}.  If $M \simeq
\mathbb{Z}^{n}$ then set as usual $N = \Hom(M,\mathbb{Z})$. A  rational finite polyhedral cone
$\sigma \subseteq M_{\mathbb{R}} = M \otimes \mathbb{R}$ is called
\emph{Gorenstein} if there exits $n_{\sigma} \in N$ with $\langle v,
n_{\sigma} \rangle = 1$ for all primitive generators $v \in M$ of rays
of $\sigma$.  This means that the affine toric variety $X_{\sigma}$ is
Gorenstein.

The cone $\sigma$ is called \emph{reflexive} if the dual cone
$\sigma^{\vee}$ is also Gorenstein. Let $m_{\sigma^{\vee}} \in M$ be
the determining lattice point. The number $r= \langle
m_{\sigma^{\vee}}, n_{\sigma} \rangle$ is the \emph{index} of the
reflexive cone $\sigma$.

A polytope in $M_{\mathbb{R}} = M \otimes \mathbb{R}$ is called a
lattice polytope if its set of vertices is in $M$.  Let $\Delta_{1},
\dots , \Delta_{r} \subseteq L_{\mathbb{R}}$ be lattice polytopes in a
rank $d$ lattice $L$. Let $M = L \oplus \mathbb{Z}^{r}$, where
$\{\epsilon_{1}, \dots ,\epsilon_{r}\}$ is the standard basis for
$\mathbb{Z}^{r}$. The cone $$ \sigma = \{(\lambda_{1}, \dots ,
\lambda_{r}, \lambda_{1}x_{1} + \dots + \lambda_{r}x_{r}) \in
M_{\mathbb{R}} : \lambda_{i} \in \mathbb{R}_{\ge 0}, x_{i} \in
\Delta_{i}, i = 1, \dots, r\}$$ is called the \emph{Cayley cone}
associated to $\Delta_{1}, \dots , \Delta_{r}$. It is a Gorenstein
cone with $n_{\sigma} = \epsilon_{1}^{\ast}+ \dots
+\epsilon_{r}^{\ast}$. A reflexive Gorenstein cone of index $r$ is
\emph{completely split} if it is the Cayley cone associated to $r$
lattice polytopes.

\subsection{Moduli and Heisenberg groups } \label{heis} There is a
large amount of literature on moduli of polarized abelian varieties
starting with \cite{mu:one}. I mention here only some articles where
the surface case is studied in detail: \cite{hs:kod}, \cite{gp:equ},
\cite{gp:cal}, \cite{ms:mod}, \cite{mr:deg} and \cite{ma:fam}.  In
particular the last three are about the $(1,7)$ case and I will
comment on them in Section \ref{T7}.

Heisenberg groups are an important ingredient in the construction of
these moduli spaces. The following construction is based on
\cite{mu:one}.  Since we are dealing with surfaces I describe only the
2 dimensional case.

Let $\delta = (d_{1},d_{2})$ be a list of elementary divisors,
i.e. $d_{i}$ are positive integers and $d_{1}| d_{2}$. Set $K(\delta)
= \mathbb{Z}_{d_{1}} \oplus \mathbb{Z}_{d_{2}}$ with character group
$K(\delta)^{\ast} = \mu_{d_{1}} \times \mu_{d_{2}}$. (Here $\mu_{d} =
\mathbb{Z}_{d}^{\ast}$ are the $d$'th roots of unity.) Define the
abstract finite \emph{Heisenberg group} $H_{\delta}$ as the extension
$$
1 \to \mu_{d_{2}} \to H_{\delta} \to K(\delta) \oplus K(\delta)^{\ast}
\to 0$$ where multiplication in $\mu_{d_{2}} \oplus K(\delta) \oplus
K(\delta)^{\ast}$ is defined by $$ (\omega, \tau, \sigma) \cdot
(\omega^{\prime}, \tau^{\prime}, \sigma^{\prime}) = (\omega \cdot
\omega^{\prime} \cdot \sigma^{\prime}(\tau), \tau + \tau^{\prime},
\sigma \cdot \sigma^{\prime}) \, .$$ If $n=d_{1}d_{2} = |K(\delta)|$
then $H_{\delta}$ has a unique $n$-dimensional irreducible
representation $V(\delta)$ in which the center $\mu_{d_{2}} \subset
\mathbb{C}^{\ast}$ acts by its natural character (\cite[Proposition
3]{mu:one}). One may realize $V(\delta)$ as the vector space of
$\mathbb{C}$ valued functions $f$ on $K(\delta)$. Then the action is
defined by $$ ((\omega, \tau, \sigma)\cdot f)(\tau^{\prime}) = \omega
\cdot \sigma(\tau^{\prime}) \cdot f(\tau + \tau^{\prime}) \, .$$ The
representation $V(\delta)$ is known as the \emph{Schr\"{o}dinger
representation} of the Heisenberg group.

\section{Overview} \label{ov} Let $T$ be an equivelar triangulation of
the torus with $n$ vertices and $X \subset \mathbb{P}^{n-1}$ the
projective Stanley-Reisner scheme associated to $T$. Let $\Gamma$ be
the sublattice of $\mathbb T$ defining $T$; i.e. $T = \{3,6\}/\Gamma$
and $G=\mathbb T/\Gamma$.

There are three for us important elements of $\mathbb{T}$ and I will
call them and their images in $G$ the \emph{principal
translations}. In the standard description (see Section~\ref{about36})
they are $$ \tau_{1} = (1, 0) \quad \tau_{2} =
\frac{1}{2}(-1,\sqrt{3}) \quad \tau_{3} = \frac{1}{2}(-1,-\sqrt{3}) \,
.$$ There is the relation $\tau_{1} + \tau_{2} + \tau_{3} = 0$. In
$\mathbb{T}$ and therefore also in $G$, any pair of them generate the
group.

The following proposition is our central observation and allows us a
natural identification of the elements of $G$ and $\vertices T$, a
fact we will use throughout. Recall that given a group $G$ with a
finite set of generators $S$ one may
construct a directed graph, called the \emph{Cayley graph}, where
the vertices are the elements of $G$. The edges are all ordered pairs
$(g,sg)$ for some $s \in S$. If one assigns a color to each element of
$S$ then the graph is colored by giving the edge $(g,sg)$  the color
of $s$. I will refer to the underlying undirected graph as a Cayley
graph as well.

\begin{proposition} \label{cg} The edge graph of $T$ is the Cayley
graph of $G$ with respect to the principal translations. In particular
the action of $G$ on $\vertices T$ is simply transitive and the set of
edges of $T$ is partitioned by the three $G$ orbits of cardinality
$n$: $$ \{ \{p,\tau_{k}(p)\} : p \in \vertices T\}$$ for $k = 1,2,3$.
\end{proposition}
\begin{proof} The edge graph of $\{3,6\}$ is clearly the Cayley graph
of $\mathbb{T}$ with respect to the principal translations.
\end{proof}

An edge in $T$ is of \emph{type $k$} if it is of the form
$\{p,\tau_{k}(p)\}$. This is the natural coloring of the Cayley
graph. The link of $p$ in $T$ will be the cycle $$ (\tau_{1}(p),
-\tau_{3}(p), \tau_{2}(p), -\tau_{1}(p), \tau_{3}(p), -\tau_{2}(p)) \,
.$$ This shows that $\tau_{i} \ne \pm \tau_{j}$ if $T$ is polyhedral.

\begin{nota} When describing the interaction between the principal
translations it will be useful to have the following convention for
the indices. If $k$ is an element in $\{1,2,3\}$ then I will use the
indices $i,j$ to represent the remaining two elements of $\{1,2,3\}
\setminus \{k\}$. I will refer to this as the $ijk$-convention.
\end{nota}

Recall from Section~\ref{simp} that the tangent space of
$\Def^{a}_{X}$ has basis $\varphi_{p,q}$, $\{p,q\} \in T$. In our new
notation the corresponding perturbation is $$
\varphi_{p,\tau_{k}(p)}(x_{m}) =
\begin{cases}
\frac{x_{m}x_{p}x_{\tau_{k}(p)}}{x_{-\tau_{i}(p)}x_{-\tau_{j}(p)}}
\quad & \text{if $\{-\tau_{i}(p),-\tau_{j}(p)\} \subseteq m$} \\ 0 &
\text{otherwise.}
\end{cases}$$ Let $t_{p,q}$ be the dual basis of coordinate functions
on $\mathbb{C}^{3n}$.

For each $p \in \vertices T$ construct the matrix
\begin{equation}
\label{minors}
\begin{bmatrix} t_{p,\tau_{1}(p)} & t_{p,\tau_{2}(p)} &
t_{p,\tau_{3}(p)}\\ t_{p,-\tau_{1}(p)} & t_{p,-\tau_{2}(p)} &
t_{p,-\tau_{3}(p)}
\end{bmatrix}
\end{equation} and take the $2 \times 2$ minors. This yields $3n$
quadratic binomials. Let $I$ to be the ideal generated by them.
\begin{theorem}[\cite{ac:def} Theorem 6.10] \label{ac} The ideal $I$
defines a versal base space in $(\mathbb{C}^{3n}, 0)$ for
$\Def^{a}_{X}$.
\end{theorem} I will misuse notation and also refer to this space as
$\Def^{a}_{X}$.

We will employ the results on binomial ideals reviewed in
Section~\ref{latticeid}. Let $L \subset \mathbb{Z}^{3n}$ be the
sublattice spanned by the exponent vectors of the generators of
$I$. Index the standard basis of $\mathbb{Z}^{3n}$,
$\varepsilon_{p,q}$, by the edges of $T$. The lattice $L$ is spanned
by the $3n$ vectors
\begin{equation}
\label{expv}
\begin{split} f_{3, p}& = \varepsilon_{p, \tau_{1}(p)} +
\varepsilon_{p, -\tau_{2}(p)} - \varepsilon_{p, \tau_{2}(p)} -
\varepsilon_{p, -\tau_{1}(p)}\\ f_{2, p}& =\varepsilon_{p,
\tau_{3}(p)} + \varepsilon_{p, -\tau_{1}(p)} - \varepsilon_{p,
\tau_{1}(p)} - \varepsilon_{p, -\tau_{3}(p)}\\ f_{1, p}&
=\varepsilon_{p, \tau_{2}(p)} + \varepsilon_{p, -\tau_{3}(p)} -
\varepsilon_{p, \tau_{3}(p)} - \varepsilon_{p, -\tau_{2}(p)} \, .
\end{split}
\end{equation} There are relations $f_{1,p}+f_{2,p}+f_{3,p} = 0$ for
each $p$ and $\sum_{p \in \vertices T} f_{k,p} = 0$ for each $k$. One
checks that indeed these are the generating relations and therefore
$\rank L = 2(n-1)$.

The $ I_{\Sat L, \rho}$ for $\rho \in (\Sat L/L)^{\ast}$ will be
minimal prime ideals for $I$. This means that each $\rho \in (\Sat
L/L)^{\ast}$ determines a component of the versal base space. I denote
these by $S_{\rho}$ and call them the \emph{main components} of
$\Def_{X}^{a}$.

Set $S$ to be the component for the trivial $\rho$, that is $S$ is
defined by the toric lattice ideal of $\Sat L$. In the recent
literature it is become normal to include non-normal varieties in the
term \emph{toric varieties}. I will also do this, thus $S$ is the germ
of an affine toric variety.  It is in general not normal. Since $\rank
\Sat L = 2(n-1)$, $\dim S = n+2$.

There are isomorphisms $\rho^{\prime}: S \simeq S_{\rho}$ as described
in Section \ref{latticeid}. By their construction, the $\rho^{\prime}$
are automorphisms of the polynomial ring restricting to the identity
on $I_{L}$ and therefore also on $I$. Thus $S \simeq S_{\rho}$ comes
from an automorphism of $\Def_{X}^{a}$. This implies that the families
over the two components are also isomorphic. So from the point of view
of deformations it is enough to study $S$.

As a toric variety the normalization of $S$, call it $\widetilde{S}$,
may be described by a rank $n+2$ lattice $M$ and a cone $\sigma^{\vee}
\subset M_{\mathbb{R}}$. To find $M$ and $\sigma^{\vee}$ we need to
find an integral $m \times 3n$ matrix $A$ (for some $m$) such that
$\ker A = \Sat L$. Then set $M = \image A$ and set $\sigma^{\vee}$ to
be the positive hull of the columns of $A$ in $M_{\mathbb{R}}$. (See
e.g. \cite{pt:tor}.) Note it is enough to find $A$ with $\rank A =
n+2$ and $L \subseteq \ker A$, since then $\rank L = \rank \ker A$ and
$\ker A$ is obviously saturated, so $\ker A = \Sat L$.  Let
$\mathbb{S}$ be the subsemigroup of $\mathbb{Z}^{m}_{\ge 0}$ generated
by the columns of $A$, thus $S$ and $\widetilde{S}$ are the germs at
$0$ of $\Spec \mathbb{C}[\mathbb{S}]$ and $\mathbb{C}[\sigma^{\vee}
\cap M]$.

There are two obvious torus actions on $\Def_{X}^{a}$ and the weights
of these actions will give us $A$. First consider the natural action
described in Section~\ref{simp} induced by automorphisms of $X$. Let
$w_{pq} \in \{(a_{0}, \dots , a_{n-1}) \in \mathbb{Z}^{n} : \sum a_{i}
= 0\}$ be the weights of this torus action on the basis
$\varphi_{pq}$. With the $ijk$-convention $$ w_{p,\tau_{k}(p)} = e_{p}
+ e_{\tau_{k}(p)} - e_{-\tau_{i}(p)} - e_{-\tau_{j}(p)}$$ for $p \in
\vertices T$, $k=1,2,3$, where $e_{p}$ are the standard basis for
$\mathbb{Z}^{n}$.

The coloring of the Cayley graph and the structure of $I$ give us
another torus action, not seen on $X$. Clearly the minors of $$
\begin{bmatrix} \lambda_{1}t_{p,\tau_{1}(p)} &
\lambda_{2}t_{p,\tau_{2}(p)} & \lambda_{3}t_{p,\tau_{3}(p)}\\
\lambda_{1}t_{p,-\tau_{1}(p)} & \lambda_{2}t_{p,-\tau_{2}(p)} &
\lambda_{3}t_{p,-\tau_{3}(p)}
\end{bmatrix}$$ also generate $I$. Thus the ${\mathbb{C}^{\ast}}^{3}$
action, $(\lambda_{1}, \lambda_{2}, \lambda_{3}) \cdot
t_{p,\tau_{k}(p)} = \lambda_{k} t_{p,\tau_{k}(p)}$ preserves $I$.

This leads to the following definition. Write the standard basis for
$\mathbb{Z}^{3} \oplus \mathbb{Z}^{n}$ as $\epsilon_{1}, \epsilon_{2},
\epsilon_{3}$ and $e_{p}, p \in \vertices T$. Let $A$ be the $(n+3)
\times 3n$ matrix with columns
\begin{equation}
\label{amatrix} A_{p,\tau_{k}(p)} = \epsilon_{k} + e_{p} +
e_{\tau_{k}(p)} - e_{-\tau_{i}(p)} - e_{-\tau_{j}(p)} \end{equation}
for $p \in \vertices T$, $k=1,2,3$. By construction $I$ is homogeneous
with respect to the multigrading with degree $t_{p,\tau_{k}(p)} =
A_{p,\tau_{k}(p)}$. It follows that  $L \subseteq \ker A$.

For each $p \in \vertices T$, the corresponding row in $A$ has a nice
description. In columns indexed by the $6$ edges having $p$ as vertex,
there is a $+1$. In the $6$ columns corresponding to edges in
$\link(p)$ we have $-1$. The other entries are $0$. Using this one
checks that $\rank A = n+2$.

Let $M= \image A$ and set $M^{\prime}$ to be the lattice
$\mathbb{Z}^{3} \oplus \{(a_{0}, \dots , a_{n-1}) \in \mathbb{Z}^{n} :
\sum a_{i} = 0\}$, the target of $A$.  For any lattice $M$ let $T_{M}=
M^{\ast}$ be the corresponding torus.  Consider the finite character
group $(M^{\prime}/M)^{\ast}$. By standard toric variety theory, see
e.g.  \cite[2.2]{fu:int}, $T_{M} =
T_{M^{\prime}}/(M^{\prime}/M)^{\ast}$ and
$$
\Spec {\mathbb{C}[M \cap \sigma^{\vee}]} = \Spec
{\mathbb{C}[M^{\prime} \cap \sigma^{\vee}]}^{(M^{\prime}/M)^{\ast}}\,
.$$ Thus the normalizations of the main components will all be
isomorphic to $$ \widetilde{S} = (\Spec \mathbb{C}[M \cap
\sigma^{\vee}],0) = (\Spec {\mathbb{C}[M^{\prime} \cap
\sigma^{\vee}]}^{(M^{\prime}/M)^{\ast}}, 0)\, .$$ Our goal is to
describe these combinatorial objects to get an as explicit as possible
description of the main components.

\section{The Stanley-Reisner ideal of $T$} \label{SR-scheme} In this
section let $I=I_{T}$ be the Stanley-Reisner ideal of the equivelar
triangulated torus. The $f$-vector of $T$ is $(n,3n,2n)$, so one may
compute the Hilbert polynomial of $A_{T}$, following \cite{sta:com},
as
$$
h_{A_{T}}(z) = n \binom{z-1}{0} + 3n \binom{z-1}{1} + 2n
\binom{z-1}{2} = nz^{2}\, .$$ This agrees with the Hilbert function
except in degree $0$. In particular one computes that
$$
\dim I_{2} = \binom{n+1}{2} - 4n = \frac{1}{2}n(n-7) \, .$$ The
minimum number of cubic generators on the other hand will depend upon
the combinatorics of $T$.

To compute the number of cubic generators consider first for every
edge $\{p,q\}$ the number $$ l_{p,q} = |\vertices (\link(\{p\},T) \cap
\link(\{q\},T)) \setminus \vertices \link(\{p,q\},T)|$$ which can be
$0$, $1$, $2$ or $3$. By symmetry $l_{p,q}$ will depend only on the
type of $\{p,q\}$, so let $l_{k}$, $k = 1,2,3$, be this common value.
\begin{lemma} \label{cub1} The minimum number of cubic generators of
$I_{T}$ is $\frac{1}{3}n (l_{1}+l_{2}+l_{3})$.
\end{lemma}
\begin{proof} A cubic monomial generator in $I_{T}$ corresponds to a
set $\{p,q,r\}$ of vertices which is a non-face, but for which every
subset is an edge. This means exactly that $\{p,q\} \in T$ and $\{r\}
\in \link(\{p\}) \cap \link(\{q\})) \setminus \link(\{p,q\})$. In the
sum $\sum_{\{p,q\} \in T} l_{p,q}$ we have counted a given such
$\{p,q,r\}$ $3$ times.
\end{proof}

The following lemma follows from a simple check.
\begin{lemma} \label{cub2} With the $ijk$-convention, $l_{k}$ is
non-zero if and only if $\tau_{k} = 2\tau_{i}$ or $\tau_{k} =
2\tau_{j}$ or $3 \tau_{k} = 0$.
\end{lemma}

\begin{proposition} The ideal $I_{T}$ is generated by quadratic and
cubic monomials. The minimum number of quadratic generators is
$\frac{1}{2}n(n-7)$.

Up to isomorphism, the $T$ which need cubic generators for $I_{T}$
have one of the following standard presentations:
$$
\begin{pmatrix} n & 2\\ 0 & 1
\end{pmatrix}, \,
\begin{pmatrix} 3 & 0\\ 0 & \frac{n}{3}
\end{pmatrix},\,
\begin{pmatrix} 3 & 1\\ 0 & \frac{n}{3}
\end{pmatrix}, \,
\begin{pmatrix} 3 & 2\\ 0 & \frac{n}{3}
\end{pmatrix} \, .$$ If $T$ is presented by $
\begin{pmatrix} n & 2\\ 0 & 1
\end{pmatrix}$, then the minimum $$ \# \text{cubic generators} =
\begin{cases} 21 \quad &\text{if $n=7$} \\ 16 \quad &\text{if $n=8$}
\\ n \quad &\text{if $n \ge 9$.}
\end{cases}$$ In the three other cases the minimum $$ \# \text{cubic
generators} =
\begin{cases} 9 \quad &\text{if $n=9$} \\ \frac{1}{3}n \quad &\text{if
$n\ge 10$.}
\end{cases}$$
\end{proposition}
\begin{proof} The only Stanley-Reisner ideal of a 2-dimensional
combinatorial manifold that needs quartic generators is the boundary
of the tetrahedron.

By Lemma \ref{cub1} and Lemma \ref{cub2} there will no cubic
generators unless some $\tau_{k} = 2\tau_{i}$ or some $3\tau_{k} =
0$. In the first case, after a $D_{6}$ movement (see
Section~\ref{about36}) we may assume $\tau_{3} = 2\tau_{1}$. In the
standard presentation this means that $(2, 1) \in \Gamma$. One checks
using e.g. \cite[Proposition 3]{bk:equ} that, for each $n$, there is
only one isomorphism class with this property and that it is
represented by the first matrix in the list. In the second case we may
assume $3\tau_{1}= 0$ and again check possibilities.

To get the number just count the $l_{k}$ in each case and use Lemma
\ref{cub1}.
\end{proof}

\section{Analysis} \label{analysis}

\subsection{The group $G$ and its principal translations} Before
proceeding with our analysis of the components I state some facts
about $G = \mathbb{T}/\Gamma$ and the $\tau_{k}$.  The proof of the
first lemma is an exercise in elementary abelian group theory.
\begin{lemma} \label{qgen} There are the following relationships
involving the principal translations.
\begin{list}{\textup{(\roman{temp})}}{\usecounter{temp}}
\item The quotient group $G/\langle \tau_{k} \rangle$ is cyclic and
the classes of $\tau_{i}$ and $\tau_{j}$ are both generators. In
particular $\lvert \tau_{i} \rvert \lvert \tau_{j} \rvert/n$ is an
integer for all $i \ne j$ in $\{1,2,3\}$.
\item Let $[g]_{i} \subseteq G$ be the coset of $\langle \tau_{i}
\rangle$ containing $g$. For any $g,h \in G$,
$$
\lvert [g]_{i} \cap [h]_{j} \rvert = \frac{\lvert \tau_{i} \rvert
\lvert \tau_{j} \rvert}{n} \, .$$
\item The number
$$
\gcd\left(\frac{n}{\lvert \tau_{i} \rvert}, \frac{n}{\lvert \tau_{j}
\rvert}\right) = \frac{n}{\lcm \left(\lvert \tau_{i} \rvert,\lvert
\tau_{j} \rvert\right)}$$ is the same for all $i \ne j$.
\end{list}
\end{lemma}

I will vary between two presentations of $G$. First there is the
\emph{standard presentation} which is the presentation in
Section~\ref{about36}. Here $G$ is a quotient of $\mathbb{T} \simeq
\mathbb{Z}^{2}$ with basis $\tau_{1}, -\tau_{3}$ by the image of
 $$
\begin{pmatrix} a & b\\ 0 & c
\end{pmatrix} \, .$$

Then we may take the \emph{symmetric presentation} where we think of
$G$ as the quotient of $\mathbb{Z}^{3}$ with basis $\tau_{1},
\tau_{2}, \tau_{3}$ by the image of a matrix
$$
R =
\begin{pmatrix} 1 & \alpha_{1,1} & \alpha_{2,1} \\ 1 & \alpha_{1,2} &
\alpha_{2,2} \\ 1 & \alpha_{1,3} & \alpha_{2,3}
\end{pmatrix} \, . $$ Of course the standard presentation is the
symmetric with
$$
R =
\begin{pmatrix} 1 & a & b \\ 1 & 0 & 0 \\ 1 & 0 & -c
\end{pmatrix} \, . $$
\begin{proposition} \label{orders} With the symmetric presentation and
the $ijk$-convention the orders of the principal translations are $$
\lvert \tau_{k} \rvert = \frac{n}{\gcd(\alpha_{1,i} - \alpha_{1,j},
\alpha_{2,i} - \alpha_{2,j})} \, .$$
\end{proposition}
\begin{proof} The determinant of $R$ is $n$. Computing it three
different ways one sees that $\gcd(\alpha_{1,i} - \alpha_{1,j},
\alpha_{2,i} - \alpha_{2,j}) | n$ for all $i \ne j$. Now $\lvert
\tau_{k} \rvert$ is the least positive $m$ with $m \varepsilon_{k} \in
\image R$.  Using Cramers rule, this is the least positive $m$ with
$m(\alpha_{1,i} - \alpha_{1,j}) \equiv m(\alpha_{2,i} - \alpha_{2,j})
\equiv 0 \mod n$. The result now follows.
\end{proof}

Here is the the abstract structure of $G$.
\begin{proposition} \label{abg} If $$ d = \gcd\left(\frac{n}{\lvert
\tau_{i} \rvert}, \frac{n}{\lvert \tau_{j} \rvert}\right)$$ for $i \ne
j$, then the elementary divisors on $G$ are $(d,n/d)$. In particular
 $$
G \simeq \mathbb{Z}_{d} \times \mathbb{Z}_{\frac{n}{d}}$$ and $G$ is
cyclic if and only if $d=1$.
\end{proposition}
\begin{proof} These invariants may be computed from the standard
presentation. We have $n = ac$, $n/\lvert \tau_{1} \rvert = c$,
$n/\lvert \tau_{2} \rvert = \gcd(a,b+c)$ and $n/\lvert \tau_{3} \rvert
= \gcd(a,b)$ (see Proposition \ref{orders}). Thus the $d$ in the
statement equals $\gcd(a,b,c)$ as it should.
\end{proof}

\subsection{$\Sat L$ and the number of main components} Let $B$ be the
$3n \times 3n$ matrix with columns the exponent vectors (\ref{expv})
of $I$. As explained in Section~\ref{ov}, $\Sat L = \ker A$. Thus
there is a commutative diagram
$$
\begin{CD} 0 @>>> L @>>> \mathbb{Z}^{3n} @>>> \Coker B @>>> 0 \\
@. @VVV @VV=V @VVV @.\\ 0 @>>> \Sat L @>>> \mathbb{Z}^{3n} @>A>> M
@>>> 0
\end{CD}$$ and the Snake Lemma yields an exact sequence
$$
0 \to \Sat L /L \to \Coker B \xrightarrow{A} M \to 0 \, .  $$ Let $d$
be as in Proposition \ref{abg}.

\begin{proposition} There is an isomorphism $\Coker B \simeq F \oplus
\mathbb{Z}_{d}$ where $F$ is free of rank $n+2$. In particular $$ \Sat
L /L \simeq \mathbb{Z}_{d}\, .$$
\end{proposition}
\begin{proof} We see from the $f_{k,p}$ described in (\ref{expv}) that
${(\Coker B )}^{\ast}$ is the set of $(\lambda_{1}, \lambda_{2},
\lambda_{3}) \in {\mathbb{C}^{\ast}}^{n} \times
{\mathbb{C}^{\ast}}^{n} \times {\mathbb{C}^{\ast}}^{n}$ satisfying
\begin{equation}
\label{rel} \frac{\lambda_{1,p}}{\lambda_{1,\tau_{1}(p)}} =
\frac{\lambda_{2,p}}{\lambda_{2,\tau_{2}(p)}} =
\frac{\lambda_{3,p}}{\lambda_{3,\tau_{3}(p)}} \text{ for all $p \in
\vertices T$.}
\end{equation} Let $\pi: {(\Coker B )}^{\ast} \to
{\mathbb{C}^{\ast}}^{n}$ be the projection on the third
factor. Clearly $\ker \pi$ equals
$$
\{(\lambda_{1}, \lambda_{2}): \lambda_{1,p}=\lambda_{1,\tau_{1}(p)},
\lambda_{2,p}=\lambda_{2,\tau_{2}(p)},\, \forall p \in \vertices T\} =
{(\mathbb{C}^{\ast})}^{\frac{n}{\lvert \tau_{1}\rvert}} \times
{(\mathbb{C}^{\ast})}^{\frac{n}{\lvert \tau_{2}\rvert}}\, .$$ To prove
the statement I will now show that $\image \pi \simeq
{(\mathbb{C}^{\ast})}^{n+2-\frac{n}{\lvert
\tau_{1}\rvert}-\frac{n}{\lvert \tau_{2}\rvert}} \times
{(\mathbb{Z}_{d})}^{\ast}$.

Let $\lambda_{3} \in\image \pi$. Choose some $p$ and let $O^{1}_{p}$
be the $\tau_{1}$ orbit of $p$. From (\ref{rel}) we get
$$
\prod _{q \in O^{1}_{p}} \frac{\lambda_{3,q}}{\lambda_{3,\tau_{3}(q)}}
= \prod _{q \in O^{1}_{p}}
\frac{\lambda_{1,q}}{\lambda_{1,\tau_{1}(q)}} = 1 \, .$$ On the other
hand, given $\lambda_{3}$ satisfying this relation, choose an
arbitrary value for $\lambda_{1,p}$ and set
$$
\lambda_{1,r\tau_{1}(p)} = \lambda_{1, p} \prod _{k=0}^{r-1}
\frac{\lambda_{3,(\tau_{3}+k\tau_{1})(p)}}{\lambda_{3,k\tau_{1}(p)}}$$
to solve (\ref{rel}).  The same is of course true for $\tau_{2}$
orbits.

Thus $\image \pi$ is the set of $\lambda \in {\mathbb{C}^{\ast}}^{n}$
with
\begin{equation}
\label{hrel} \prod_{p \in O} \frac{\lambda_{p}}{\lambda_{\tau_{3}(p)}}
= 1 \text{ for all $\tau_{1}$ and $\tau_{2}$ orbits $O$.}
\end{equation} If $\mathcal{P}_{i}$ are the orbit partitions of
$\vertices T$ by $\tau_{i}$, then $\tau_{3}$ acts transitively on
$\mathcal{P}_{1}$ and $\mathcal{P}_{2}$ (Lemma~\ref{qgen}). Thus
condition (\ref{hrel}) translates to
\begin{equation*} \prod_{p \in O} \lambda_{p} = \prod_{q \in
O^{\prime}}\lambda_{q} \text{ for all $O, O^{\prime} \in
\mathcal{P}_{1}$ and all $O, O^{\prime} \in \mathcal{P}_{2}$.}
\end{equation*} For $O \in \mathcal{P}_{i}$ let this common value be
$\mu_{i} = \prod_{p \in O}\lambda_{p}$, $i=1,2$. There is a
homomorphism $\phi : \image \pi \to {\mathbb{C}^{\ast}}^{2}$, $\lambda
\mapsto (\mu_{1}, \mu_{2})$. Clearly $\ker \phi$ is the set of
$\lambda \in {\mathbb{C}^{\ast}}^{n}$ with $\prod_{p \in
O}\lambda_{p}=1$ for all $O \in \mathcal{P}_{1}$ and all $O \in
\mathcal{P}_{2}$.

Let $R_{G}$ be the free abelian group on the elements of $G$ with the
regular $G$ action. Set $R_{G}^{\tau_{i}}$ to be the invariant
sublattice under the action of $\langle \tau_{i} \rangle$. We may
realize $\ker \phi$ as the kernel of the projection $R_{G}^{\ast} \to
(R_{G}^{\tau_{1}} + R_{G}^{\tau_{2}})^{\ast}$. Now
\begin{align*}\rank (R_{G}^{\tau_{1}} + R_{G}^{\tau_{2}}) &= \rank
R_{G}^{\tau_{1}} + \rank R_{G}^{\tau_{2}} - \rank (R_{G}^{\tau_{1}}
\cap R_{G}^{\tau_{2}})\\ &= \frac{n}{\lvert
\tau_{1}\rvert}+\frac{n}{\lvert \tau_{2}\rvert}-1
\end{align*} by Lemma~\ref{qgen}. Thus $\ker \phi \simeq
{(\mathbb{C}^{\ast})}^{n+1-\frac{n}{\lvert
\tau_{1}\rvert}-\frac{n}{\lvert \tau_{2}\rvert}}$.

There is one relation between the $\mu_{i}$ namely
$$
\prod_{p \in \vertices T}\lambda_{p} = \mu_{1}^{\frac{n}{\lvert
\tau_{1}\rvert}} = \mu_{2}^{\frac{n}{\lvert \tau_{2}\rvert}}\, .$$ But
$d = \gcd\left( n/\lvert \tau_{1} \rvert, n/\lvert \tau_{2}\rvert
\right)$ so $\image \phi \simeq \mathbb{C}^{\ast} \times
{(\mathbb{Z}_{d})}^{\ast}$. Adding this up gives the result.
\end{proof}

In terms of the structure of $\Def_{X}^{a}$ one has
\begin{corollary} If $d$ is the first elementary divisor of $G$ then
the number of main components in $\Def_{X}^{a}$ is $d$.
\end{corollary}

\subsection{The group $(M^{\prime}/M)^{\ast}$} I will compute
$(M^{\prime}/M)^{\ast}$. This is of interest in itself, but will also
be important for out study of $S$ in Section~\ref{modulisec}.

From the weight matrix $A$ (\ref{amatrix}) we see that ${\left(
M^{\prime}/M \right)}^{\ast} \subset {\mathbb{C}^{\ast}}^{3} \times
({\mathbb{C}^{\ast}}^{n}/\mathbb{C}^{\ast})$ is defined by
\begin{equation}
\label{ml}
\mu_{k}\frac{\lambda_{p}\lambda_{\tau_{k}(p)}}{\lambda_{-\tau_{i}(p)}
\lambda_{-\tau_{j}(p)}} = 1
\end{equation} where $(\mu_{1}, \mu_{2}, \mu_{3}) \in
{\mathbb{C}^{\ast}}^{3}$ and $\lambda \in
{\mathbb{C}^{\ast}}^{n}/\mathbb{C}^{\ast}$ has coordinates indexed by
$\vertices T$. I will first solve these equations for $\tau_{k} \in
\mathbb{T}$ and $p \in \vertices \{3, 6\}$ and then see what happens
in the quotient. For this purpose we need the following quadratic
parabolic function.
\begin{definition} Define $q: \mathbb{Z}^{2} \to \mathbb{Z}$ by
$$
q(x,y) = \frac{1}{2}(x^{2} + y^{2} -2xy - x- y) =
\frac{1}{2}((x-y)^{2}- (x+y))\, .$$
\end{definition}

\begin{lemma} \label{qcomp} For $q$ there are the equalities $$
q(mx,my) = m q(x,y) + \frac{1}{2}m(m-1) \, (x-y)^{2}$$ and
$$
q(x+z,y+w) = q(x,y) + q(z,w) + (x-y)(z-w) \, .$$
\end{lemma}

\begin{lemma} Choose an origin $0$ in $\vertices \{3,6\}$ and
$\lambda_{0}, \lambda_{k}, \mu_{k} \in \mathbb{C}^{\ast}$, $k=1,2,3$
satisfying
\begin{equation}
\label{con} \lambda_{0}^{3} = \frac{\lambda_{1}\lambda_{2}\lambda_{3}}
{\mu_{1}\mu_{2}\mu_{3}} \, .
\end{equation} After these choices, a solution for (\ref{ml}) with
$\tau_{k} \in \mathbb{T}$ and $p \in \vertices \{3, 6\}$ is unique and
given for $p = \alpha \tau_{1}(0) + \beta \tau_{2}(0) + \gamma
\tau_{3}(0)$ by
\begin{equation}
\label{sol} \lambda_{p} = \lambda_{0}^{1-\alpha - \beta - \gamma}
\lambda_{1}^{\alpha} \lambda_{2}^{\beta} \lambda_{3}^{\gamma}\,
\mu_{1}^{q(\beta, \gamma)} \mu_{2}^{q(\alpha, \gamma)}
\mu_{3}^{q(\alpha, \beta)} \, .
\end{equation}
\end{lemma}
\begin{proof} We first check that the (\ref{sol}) is well
defined. Assume $p= \alpha \tau_{1}(0) + \beta \tau_{2}(0) + \gamma
\tau_{3}(0) = a \tau_{1}(0) + b\tau_{2}(0) + c \tau_{3}(0)$, then
$\alpha -a = \beta -b = \gamma -c = \delta$ for some $\delta$. Now
$q(x + \delta, y + \delta) = q(x,y) - \delta$ so the right hand side
in (\ref{sol}) is changed by multiplication with
\begin{equation*} {\left(\frac{\lambda_{1}\lambda_{2}\lambda_{3}}
{\lambda_{0}^3\, \mu_{1}\mu_{2}\mu_{3}}\right)}^{\delta}
\end{equation*} which equals $1$ by condition (\ref{con}).

For the uniqueness consider a vertex $p$ and its link with vertices
$q_{1}, \dots, q_{6}$. One checks directly that if the $\mu_{k}$,
$\lambda_{p}$ and $3$ of the $\lambda_{q_{i}}$ are given, then the
relations (\ref{ml}) determine the other $3$ $\lambda_{q_{i}}$. Now
beginning in the origin and working outward we see that all
$\lambda_{p}$ are determined by $\lambda_{0}$ and
$\lambda_{\tau_{k}(0)}$, $k=1,2,3$.

Finally we must check that these $\lambda_{p}$ actually are
solutions. We do this only for $k=1$. After plugging (\ref{sol}) into
(\ref{ml}) and some obvious cancellations we arrive at $$ \frac{
\lambda_{1}\lambda_{2}\lambda_{3}\, \mu_{1}^{2q(\beta, \gamma) + 1}
\mu_{2}^{q(\alpha+1, \gamma)} \mu_{3}^{q(\alpha +1,
\beta)}}{\lambda_{0}^{3} \, \mu_{1}^{q(\beta-1, \gamma)}
\mu_{1}^{q(\beta, \gamma-1)} \mu_{2}^{q(\alpha, \gamma-1)}
\mu_{3}^{q(\alpha, \beta-1)} } \, . $$ Now $q(x+1,y) = q(x,y) + x -y$
and $q(x-1,y) = q(x,y) + y-x +1$, so this reduces further to
$$
\frac{ \lambda_{1}\lambda_{2}\lambda_{3} }{\lambda_{0}^{3} \,
\mu_{1}\mu_{2} \mu_{3} } = 1\, . $$
\end{proof}

Now take the quotient by $\Gamma$. With the notation of the symmetric
presentation let $\alpha_{t,1} \tau_{1} + \alpha_{t,2} \tau_{2} +
\alpha_{t,1}\tau_{3}$, $t=1,2$, be generators of $\Gamma$ and set
$r_{t} = \alpha_{t,1} \tau_{1}(0) + \alpha_{t,2} \tau_{2}(0) +
\alpha_{t,1}\tau_{3}(0)$.
\begin{lemma} \label{mcon} The character group ${\left( M^{\prime}/M
\right)}^{\ast} \subset {\mathbb{C}^{\ast}}^{3} \times
({\mathbb{C}^{\ast}}^{n}/\mathbb{C}^{\ast})$ consists of the solutions
(\ref{sol}) under the condition (\ref{con}) with $$ \lambda_{0} =
\lambda_{r_{1}} = \lambda_{r_{2}}= 1$$ and $$
\mu_{1}^{A(\alpha_{t,2}-\alpha_{t,3})}
\mu_{2}^{B(\alpha_{t,3}-\alpha_{t,1})}
\mu_{3}^{C(\alpha_{t,1}-\alpha_{t,2})} = 1$$ for $t=1,2$ and all
integers $A,B,C$ with $A+B+C =0$.
\end{lemma}
\begin{proof} Setting $\lambda_{0} = 1$ corresponds to the second
factor being ${\mathbb{C}^{\ast}}^{n}/\mathbb{C}^{\ast}$. The first
condition is clearly necessary.

For the solution to be valid modulo $\Gamma$ we must have
$\lambda_{m_{1}r_{1} + m_{2}r_{2}+p} = \lambda_{p}$ for all integers
$m_{i}$ and $p \in \vertices \{3,6\}$.  Consider first
\begin{multline*} \lambda_{m_{1}r_{1} + m_{2}r_{2}} =
\lambda_{r_{1}}^{m_{1}} \lambda_{r_{2}}^{m_{2}}\\ \cdot \left(
\mu_{1}^{(\alpha_{1,2}-\alpha_{1,3}) (\alpha_{2,2}-\alpha_{2,3})}
\mu_{2}^{(\alpha_{1,3}-\alpha_{1,1})(\alpha_{2,3}-\alpha_{2,1})}
\mu_{3}^{(\alpha_{1,1}-\alpha_{1,2})(\alpha_{2,1}-\alpha_{2,2})}
\right)^{m_{1}m_{2}} \\ \cdot\prod_{t=1}^{2} \left(
\mu_{1}^{(\alpha_{t,2}-\alpha_{t,3})^{2}}
\mu_{2}^{(\alpha_{t,3}-\alpha_{t,1})^{2}}
\mu_{3}^{(\alpha_{t,1}-\alpha_{t,2})^{2}}
\right)^{\frac{1}{2}m_{t}(m_{t}-1)}
\end{multline*} by Lemma \ref{qcomp}. The two conditions in the
statement imply that this expression equals $1$.

Now, in general, if $r = \alpha \tau_{1}(0) + \beta \tau_{2}(0) +
\gamma \tau_{3}(0)$ and $p = a \tau_{1}(0) + b\tau_{2}(0) +
c\tau_{3}(0)$ then
\begin{equation}
\label{gen} \lambda_{r+p} = \lambda_{r}\lambda_{p} \,
\mu_{1}^{(\beta-\gamma) (b-c)} \mu_{2}^{(\gamma-\alpha)(c-a)}
\mu_{3}^{(\alpha-\beta)(a-b)}
\end{equation} by Lemma \ref{qcomp}.

Let $r = m_{1}r_{1} + m_{2}r_{2}$ and $A=b-c, B=c-a, C=a-b$ in the
(\ref{gen}). Set $\beta_{t,ij} = \alpha_{t,i}-\alpha_{t,j}$ to shorten
notation. Then the factor involving the $\mu_{i}$ in (\ref{gen})
becomes $$ \mu_{1}^{A(m_{1}\beta_{1,2,3} + m_{2}\beta_{2,2,3})}
\mu_{2}^{B(m_{1}\beta_{1,3,1} + m_{2}\beta_{2,3,1})}
\mu_{3}^{C(m_{1}\beta_{1,1,2} + m_{2}\beta_{2,1,2})} = 1$$ by the
second condition. Thus for all $p$, $\lambda_{m_{1}r_{1} +
m_{2}r_{2}+p} = \lambda_{p}$. Choosing $m_{1}=1, m_{2}=0$ and vice
versa gives the necessity of the second condition.
\end{proof}

\begin{proposition} \label{ext} There is an extension $$ 1 \to
G^{\ast} \to {\left( M^{\prime}/M \right)}^{\ast} \to G^{\ast} \times
{(\mathbb{Z}_{d})}^{\ast} \to 1\, .$$ In particular $|M^{\prime}/M| =
n^2d$.
\end{proposition}
\begin{proof} Consider the projection on the first factor of
${\mathbb{C}^{\ast}}^{3} \times {\mathbb{C}^{\ast}}^{n-1}$ restricted
to ${\left( M^{\prime}/M \right)}^{\ast}$. I claim the kernel is
$G^{\ast}$. Indeed, if we set $\mu_{i} = 1$ in the conditions of Lemma
\ref{mcon} we are left with $$ \lambda_{1}\lambda_{2} \lambda_{3} =
\lambda_{1}^{\alpha_{1,1}} \lambda_{2}^{\alpha_{1,2}}
\lambda_{3}^{\alpha_{1,3}} = \lambda_{1}^{\alpha_{2,1}}
\lambda_{2}^{\alpha_{2,2}} \lambda_{3}^{\alpha_{2,3}} = 1 \, .$$ Thus
to prove the statement we must show that the image of the projection
is the character group of $\mathbb{Z}_{d} \times \mathbb{Z}_{d} \times
\mathbb{Z}_{n/d}$.

Consider the relations among the $\mu_{i}$ described in Lemma
\ref{mcon}.  There are $4$ generating relations corresponding to $(A,
B, C) = (1,-1,0)$ and $(0,-1,1)$. We may use the standard presentation
of $G$ to compute them. They are $$ \mu_{2}^{a} =
\mu_{2}^{a}\mu_{3}^{a} = \mu_{1}^{c}\mu_{2}^{b+c} =
\mu_{2}^{b+c}\mu_{3}^{b} = 1\, .$$ One may compute the $\gcd$ of
minors and find that the elementary divisors of
\begin{equation}
\label{z3}
\begin{pmatrix} 0 & 0 & c & 0\\ a & 0 & b+c & b+c \\ 0 & a & 0 & b
\end{pmatrix}
\end{equation} are $(d,d,n/d)$.
\end{proof}

\subsection{The cone $\sigma^{\vee}$} \label{conesec} Let $N^{\prime}
\subseteq N$ be the dual lattices of $M \subseteq M^{\prime}$ and
$\sigma$ the dual cone of $\sigma^{\vee}$. Recall that $\sigma^{\vee}$
is the positive hull of the columns of $A$ in $M_{\mathbb{R}}$ and
that the columns of $A$ are
$$
A_{p,\tau_{k}(p)} = \epsilon_{k} + e_{p} + e_{\tau_{k}(p)} -
e_{-\tau_{i}(p)} - e_{-\tau_{j}(p)}$$ for $k = 1, 2,3$ and $p \in
\vertices T$. We will need the easily checked lemma.
\begin{lemma} \label{i-orbit} If $i \ne j$ and $O$ is a $\tau_{i}$
orbit in $\vertices T$, then $ \sum_{q \in O} A_{q,\tau_{j}(q)} =
|\tau_{i}|\epsilon_{j}$.
\end{lemma}

The matrix A has the nice property that the columns generate the rays
of $\sigma^{\vee}$.
\begin{lemma} \label{rays} Each column of $A$ is a primitive generator
in $M$ for a ray of $\sigma^{\vee}$, thus $\sigma^{\vee}$ has $3n$
rays.
\end{lemma}
\begin{proof} For each edge $\{p,\tau_{k}(p)\}$ of $T$ let
$u_{p,\tau_{k}(p)} \in N^{\prime}$ be $2(\epsilon_{1}^{\ast} +
\epsilon_{2}^{\ast} + \epsilon_{2}^{\ast}) -(e^{\ast}_{p} +
e^{\ast}_{\tau_{k}(p)})$. If $A_{q,\tau_{l}(q)}$ is a column of $A$,
then $0 \le \langle A_{q,\tau_{l}(q)}, u_{p,\tau_{k}(p)} \rangle \le 4
$ and equals $0$ if and only if $p=q$ and $k=l$. Thus
$u_{p,\tau_{k}(p)} \in \sigma$ and it defines the $1$ dimensional face
spanned by the column $A_{p,\tau_{k}(p)}$.
\end{proof}

\begin{proposition} The cone $\sigma^{\vee} \subseteq M_{\mathbb{R}}$
is a Gorenstein cone. It is a Cayley cone associated to 3 lattice
polytopes.
\end{proposition}
\begin{proof} If $n_{\sigma^{\vee}} = \epsilon_{1}^{\ast} +
\epsilon_{2}^{\ast} + \epsilon_{2}^{\ast}$, then clearly $\langle
n_{\sigma^{\vee}}, A_{q,\tau_{l}(q)}\rangle = 1$ so $\sigma^{\vee}$ is
Gorenstein by Lemma \ref{rays}. It is a Cayley cone by
\cite[Proposition 2.3]{bn:com}.
\end{proof}

\begin{remark} It seems a difficult but interesting combinatorial
  problem to determine the type of polytopes these three
  are. They vary with the combinatorics of $T$. For example they are in
  general not $n-1$ dimensional though their Minkowski sum is. If they
  are $n-1$ dimensional then they must be simplices since they have
  $n$ vertices. This is the case when the corresponding $\tau_{i}$ has
  order $n$.
\end{remark}

In fact $\sigma^{\vee}$ has a finer Cayley structure. By a \emph{Cayley
   structure} on  $\sigma^{\vee}$ I mean some set of lattice polytopes  $\Delta_{1}, \dots
 , \Delta_{r}$ such that  is $\sigma^{\vee}$ is the Cayley cone
 associated to $\Delta_{1}, \dots , \Delta_{r}$. First partition 
each of the sets of type $k$ columns $\{A_{p,\tau_{k}(p)}: p
\in\vertices T\}$ into its $\tau_{k}$ orbits. This partitions the set
of all $3n$ columns into $r$ cells where $$ r= \frac{n}{|\tau_{1}|} +
\frac{n}{|\tau_{2}|} +\frac{n}{|\tau_{3}|}\, .$$ Index these cells
$o_{1}, \dots ,o_{r}$ and view $\mathbb{Z}^{r}$ as the free abelian
group on the $o_{i}$. Let $\beta_{i}$ be the standard basis element of
$\mathbb{R}^{r}$ corresponding to $o_{i}$. The orbit $o_{i}$ is of
type $k$ if it is a $\tau_{k}$ orbit of type $k$ columns.

Now define the vectors
\begin{equation} \label{them} m_{p} = \sum_{k=1}^{3}(e_{\tau_{k}(p)} -
e_{-\tau_{k}(p)}),\quad p \in \vertices T\, .
\end{equation} 
Since
\begin{align*} A_{-\tau_{k}(p),p} &= \epsilon_{k} + e_{p} +
e_{-\tau_{k}(p)} - e_{(-\tau_{i}-\tau_{k})(p)} -
e_{(-\tau_{j}-\tau_{k})(p)}\\ &= \epsilon_{k} + e_{p} +
e_{-\tau_{k}(p)} - e_{\tau_{j}(p)} - e_{\tau_{i}(p)}
\end{align*} we have $m_{p} = A_{p,\tau_{k}(p)}- A_{-\tau_{k}(p),p}$
for all $k=1,2,3$. Thus $m_{p} \in M$, so define $M^{\prime\prime}
\subset M$ to be the sublattice spanned by the $m_{p}$.

Let $\tilde{\Delta}$ be the \emph{support} of the Gorenstein cone
$\sigma^{\vee}$, i.e. the polytope $\{x \in \sigma^{\vee} : \langle
n_{\sigma^{\vee}}, x \rangle = 1\}$.

\begin{theorem} There is an exact sequence
$$
0 \to M^{\prime\prime} \to M \to \mathbb{Z}^{r} \to 0$$ where the last
map takes $A_{p,\tau_{k}(p)} \mapsto \beta_{i}$ if $A_{p,\tau_{k}(p)}
\in o_{i}$. This projection maps $\tilde{\Delta}$ surjectively on the
convex hull of $\{\beta_{1}, \dots , \beta_{r}\}$ and therefore
determines a Cayley structure of length $r$ on $\sigma^{\vee}$.
\end{theorem}
\begin{proof} We must show that the application $A_{p,\tau_{k}(p)}
\mapsto \beta_{i}$ gives us a well-defined morphism $M \to
\mathbb{Z}^{r}$. This would follow from the following claim: $$
\sum_{k=1}^{3}\sum_{p \in \vertices T} \alpha_{k,p}A_{p,\tau_{k}(p)} =
0 \quad \Longrightarrow \quad \sum_{A_{p,\tau_{k}(p)} \in o_{i}}
\alpha_{k,p} = 0, \quad i = 1, \dots , r \, .$$ Assume $\sum\sum
\alpha_{k,p}A_{p,\tau_{k}(p)} = 0 $ and that $o_{i}$ is of type
$k$. We have
\begin{multline*} 0 =
\sum_{m=0}^{|\tau_{k}|-1}m\tau_{k}(\sum_{l=1}^{3}\sum_{p \in \vertices
T} \alpha_{l,p}A_{p,\tau_{l}(p)}) \\ = \sum_{l=1}^{3}\sum_{p \in
\vertices T} \alpha_{l,p}\sum_{m=0}^{|\tau_{k}|-1}
A_{m\tau_{k}(p),(m\tau_{k}+\tau_{l})(p)} \\ = (\sum
\alpha_{i,p})|\tau_{k}| \epsilon_{i} + (\sum \alpha_{j,p})|\tau_{k}|
\epsilon_{j} +\sum \alpha_{k,p}\sum_{m=0}^{|\tau_{k}|-1}
A_{m\tau_{k}(p),(m+1)\tau_{k}(p)}
\end{multline*} by Lemma \ref{i-orbit}. The right hand term cannot
cancel the $\epsilon_{i}$ or $\epsilon_{j}$ term, so must also
vanish. Reindex the orbits of type $k$ so they are $o_{1}, \dots ,
o_{n/|\tau_{k}|}$. Now
$$
\sum_{p} \alpha_{k,p}\sum_{m=0}^{|\tau_{k}|-1}
A_{m\tau_{k}(p),(m+1)\tau_{k}(p)} = \sum_{i = 1}^{n/|\tau_{k}|}
(\sum_{A_{p,\tau_{k}(p)} \in o_{i}} \alpha_{k,p})
(\sum_{A_{p,\tau_{k}(p)} \in o_{i}} A_{p,\tau_{k}(p)})$$ so we must
show that the $\{\sum_{A_{p,\tau_{k}(p)} \in o_{i}} A_{p,\tau_{k}(p)}:
i = 1, \dots n/|\tau_{k}|\}$ is linearly independent.

Let $G_{k}$ be the subgroup of $G$ generated by $\tau_{k}$ acting on
$\mathbb{Z}^{n+1}$, with basis $\epsilon_{k}$ and $e_{p}, p \in
\vertices T$, with $\tau_{k}(\epsilon_{k}) = \epsilon_{k}$ and
$\tau_{k}(e_{p}) = e_{\tau_{k}(p)}$. Let $[p]$ denote the $G_{k}$
orbit of $p$ in $\vertices T$. The invariant sublattice
$(\mathbb{Z}^{n+1})^{G_{k}}$ has rank $n/|\tau_{k}|$ and is spanned by
$\epsilon_{k}$ and $\beta_{[p]} = \sum_{q \in [p]}e_{q}$. (If $n =
|\tau_{k}|$ then of course $\beta_{[p]} = 0$ and
$(\mathbb{Z}^{n+1})^{G_{k}}$ is spanned by $\epsilon_{k}$.)

Each $\sum_{A_{p,\tau_{k}(p)} \in o_{i}} A_{p,\tau_{k}(p)} \in
(\mathbb{Z}^{n+1})^{G_{k}}$. If $A_{p,\tau_{k}(p)} \in o_{i}$ then one
computes
\begin{equation}
\label{invbasis} \sum_{A_{q,\tau_{k}(q)} \in o_{i}} A_{q,\tau_{k}(q)}
= |\tau_{k}|\epsilon_{k} + 2\beta_{[p]} - \beta_{[-\tau_{i}(p)]} -
\beta_{[-\tau_{j}(p)]} \, .
\end{equation} Now both $G_{i}$ and $G_{j}$ act transitively on the
set of $G_{k}$ orbits of $\vertices T$ by $\tau_{i}([p]) =
[\tau_{i}(p)]$ and similarly for $G_{j}$ (see Lemma \ref{qgen}). So,
after choosing some $p_{0} \in \vertices T$ and setting
$\bar{\tau}_{i}$ to be the class of $\tau_{i}$ in $G/G_{k}$, index the
basis by $\beta_{[p]} = \beta_{m}$ if $[p] =
m\bar{\tau}_{i}([p_{0}])$. Moreover $[-\tau_{j}(p)] = [(\tau_{k}+
\tau_{i})(p)] = [\tau_{i}(p)]$. Thus, with new indices, the vectors in
(\ref{invbasis}) become $$ |\tau_{k}|\epsilon_{k} - \beta_{m-1} +
2\beta_{m} - \beta_{m+1}, \quad m = 0, \dots ,\frac{n}{|\tau_{k}|}
-1$$ (indexed cyclicly) and this is a linearly independent set.

Since $m_{p} = A_{p,\tau_{k}(p)}- A_{-\tau_{k}(p),p}$, for all $k$,
they generate the kernel of $M \to \mathbb{Z}^{r}$. The statement
about convex hulls follows from the description of the map. The
statement about Cayley structures is again \cite[Proposition
2.3]{bn:com}.
\end{proof}

\begin{remark} In \cite{bn:com} we are told how to find the $r$
polytopes making up the Cayley structure. The support $\tilde{\Delta}$
is the convex hull of the columns of $A$. Choose some element in each
$o_{i}$ and call it $E_{i}$ and a basis $E_{r+1}, \dots , E_{n+2}$ for
$M^{\prime\prime}$. Thus $\{E_{1}, \dots , E_{n+2}\}$ is a basis for
$M$. Let $E_{i}^{\ast}$ be the dual basis and set for $i = 1, \dots ,
r$ $$ \tilde{\Delta}_{i} = \{x \in \tilde{\Delta} : \langle x,
E_{j}^{\ast} \rangle = 0 \text{ for $j \in \{1, \dots, r\} \setminus
\{i\}$} \} \, .$$ Write $\tilde{\Delta}_{i} = \Delta_{i} \times E_{i}$
where $\Delta_{i}$ is a lattice polytope in
$M^{\prime\prime}_{\mathbb{R}}$. The cone $\sigma^{\vee}$ is the
Cayley cone associated to $\Delta_{1}, \dots \Delta_{r}$.
\end{remark}

\section{Deformations} \label{def}

We may pull back the family over $S$ to the normalization
$\widetilde{S}$, which is finite and generically injective over
$S$. Thus if we are only interested in which fibers occur, then we may
as well work on $\widetilde{S}$.

Let $R$ be the local ring of $\Def_{X}^{a}$. In the proof of Theorem
\ref{ac} in \cite{ac:def} we constructed a \emph{local formal model}
of the versal family over $\Def_{X}^{a}$. That is a collection
$\mathcal{U}_{p}$, $p \in \vertices T$, of affine schemes and
deformations $\mathcal{U}_{p} \to \Def_{X}^{a}$ of $U_{p}$ such that
over $R_{n} = R/\mathfrak{m}^{n+1}$, the $\mathcal{U}_{p}
\times_{\Def^{a}_{X}} \Spec R_{n}$ could be glued to form a formal
versal deformation $X_{n} \to \Spec R_{n}$. Thus if $\mathcal{X} \to
\Def_{X}^{a}$ is a formally versal deformation, then $$ (\mathcal{X}
\times_{\Def^{a}_{X}} \Spec R_{n})|_{U_{p}} \simeq \mathcal{U}_{p}
\times_{\Def^{a}_{X}} \Spec R_{n}$$ as formal deformations of $U_{p}$.

We may therefore apply the following application of Artin
approximation.
\begin{theorem} \label{aa} Let $R$ be a local $k$-algebra and assume
$X \to \Spec R$ and $Y \to \Spec R$ are two deformations of $X_{0}$
with isomorphic associated formal deformations. Then $X \setminus
X_{0}$ is smooth near $X_{0}$ if and only if $Y \setminus X_{0}$ is
smooth near $X_{0}$.
\end{theorem}
\begin{proof} Let $x \in X_{0}$. By assumption
$\hat{\mathcal{O}}_{X,x} \simeq \hat{\mathcal{O}}_{Y,x}$, thus by the
variant of Artin approximation theorem in \cite[Corollary
2.6]{ar:alg}, $X$and $Y$ are locally isomorphic for the \'{e}tale
topology near $x$.
\end{proof}

We know the $\mathcal{U}_{p}$ in detail - see \cite[Proof of
6.10]{ac:def}. Label the coordinates of $\mathbb{P}^{n-1}$ by $x_{p}$,
$p \in \vertices T$. On $U_{p}$ denote the $6$ coordinates by
$y_{p,\pm k} = x_{\pm \tau_{k}(p)}/x_{p}$, $k= 1,2,3$. Then
$\mathcal{U}_{p}$ is defined by the ideal generated by the $9$
equations
\begin{equation}
\label{Up}
\begin{split} &y_{p, \mp i}y_{p, \mp j} +t_{p,\pm \tau_{k}(p)}
y_{p,\pm k} \quad k = 1,2,3 \\ &y_{p, k}y_{p, -k} -t_{p,
-\tau_{i}(p)}t_{p, \tau_{j}(p)}\quad k = 1,2,3 \, .
\end{split}
\end{equation} Recall that $t_{p, -\tau_{i}(p)}t_{p, \tau_{j}(p)} =
t_{p, \tau_{i}(p)}t_{p, -\tau_{j}(p)}$ in $\Def^{a}_{X}$ so the last
equation makes sense.

Note that if the coordinates of $\Def_{X}^{a}$ corresponding to edges
of the same type are equated, $t_{p, \tau_{k}(p)} = t_{q,
\tau_{k}(q)}$ for all $p,q \in \vertices T$, the minors of the
matrices (\ref{minors}) vanish. This defines a smooth $3$-dimensional
subspace $\mathcal{M}$ of $\Def_{X}^{a}$. Recall from
Section~\ref{simp} that the action of $G$ on $\Def_{X}^{a}$ is the
same as the action on the edges of $T$, i.e. $g \cdot t_{pq} =
t_{g(p),g(q)}$. It follows immediately that $\mathcal{M} =
{(\Def_{X}^{a})}^{G}$.

In toric terms we may describe $\mathcal{M}$ this way. Consider the
projection on the first factor $p_{1}: M^{\prime} \to
\mathbb{Z}^{3}$. The restriction to $M$ is surjective and the induced
map $M_{\mathbb{R}} \to \mathbb{R}^{3}$ maps $\sigma^{\vee}$ and
$\mathbb{S}$ onto the positive octant. Thus we have a closed embedding
of $\mathbb{C}^{3}= \Spec \mathbb{C}[\mathbb{Z}_{\ge 0}^{3}]$ into
both $\Spec \mathbb{C}[M \cap \sigma^{\vee}]$ and $\Spec
\mathbb{C}[\mathbb{S}]$. It follows that $\mathcal{M}$ lies in the
toric component $S$. The inclusion $\mathcal{M} \subset S$ is the
surjection $\mathbb{C}[\mathbb{S}] \to \mathbb{C}[\mathbb{Z}_{\ge
0}^{3}]$ induced by the projection $p_{1}: M \to \mathbb{Z}^{3}$.

Let $K$ be the image of $M$ under the projection on the second
factor $$ M^{\prime}\to \{(a_{0}, \dots , a_{n-1}) \in \mathbb{Z}^{n}
: \sum a_{i} = 0\}\, .$$ This yields an inclusion $T_{K} \subseteq
T_{M}$ which corresponds to the natural
$(\mathbb{C}^{*})^{n}/\mathbb{C}^{*}$ action on $\Def_{X}^{a}$ induced
by the action on $X$ (see Section~\ref{simp}). Now $T_{M}$ and
therefore $T_{K}$ are subspaces of both $S$ and $\widetilde{S}$. We
will need the following easily proven lemma.

\begin{lemma} \label{tktm} Every point in $T_{M}$ is in a $T_{K}$
orbit of a point in $T_{M} \cap \mathcal{M}$.
\end{lemma}

In terms of deformations
\begin{lemma} \label{tkorb} The fibers over a $T_{K}$ orbit in $S$ (or
$\widetilde{S}$) are isomorphic.
\end{lemma}
\begin{proof} This probably follows from general principles since the
action of $T_{K}$ is induced by $\exp$ of the Lie algebra action of
$H^{0}(X, \Theta_{X})$ on $T^{1}_{X}$. One sees this directly though
by noting that $T_{K}$ acts as automorphisms on $\mathcal{U}_{p}$
compatible with the formal gluing (see \cite[Proof of 6.10]{ac:def}).
\end{proof}

\begin{theorem} \label{disc} The main components are the only smoothing components
of $\Def_{X}^{a}$ and the discriminant of $\widetilde{S}$ is
$\widetilde{S} \setminus T_{M}$.
\end{theorem}
\begin{proof} By Theorem \ref{es}, on a non-main component some $t_{p,
\tau_{k}(p)} = 0$. After looking at the equations (\ref{Up}) of the
local formal model we conclude that $\mathcal{U}_{p}$ will be
singular, in fact reducible. But then by Theorem \ref{aa},
$\mathcal{X}$ cannot contain a smooth fiber over this component. By
standard toric geometry the same argument applies to fibers over
$\widetilde{S} \setminus T_{M}$.

It remains to show that the fibers over $T_{M}$ are smooth. Consider a
one parameter sub-family over $C \subset \mathcal{M} \subseteq S$. We
may assume $C$ is given by $t_{p, \tau_{1}(p)} = at, t_{p,
\tau_{2}(p)} = bt, t_{p, \tau_{3}(p)} = ct$ for some $a,b,c \in
\mathbb{C}$. Plug this into the equations (\ref{Up}) and one sees that
if $abc \neq 0$ the charts over this curve have an isolated
singularity at $0$. Thus if $abc \neq 0$, Theorem \ref{aa} and generic
smoothness imply that $\mathcal{X}_{|C}$ is a smoothing. Since it is a
one-parameter smoothing the nearby fibers will all be smooth. Thus
$T_{M} \cap \mathcal{M}$ has only smooth fibers, but then by
Lemma~\ref{tktm} and Lemma~\ref{tkorb} the same is true for $T_{M}$.
\end{proof}

\section{Moduli}\label{modulisec} I will construct the Heisenberg
group $H_{(d,n/d)}$ from $G$. After choosing an origin in $\{3,6\}$
there is a one to one correspondence $G \to \vertices T$ given by
$\tau \mapsto \tau(0)$.  As before label the coordinates of
$\mathbb{P}^{n-1}$ by $x_{p}$, $p \in \vertices T$.

The group $G \subseteq \Aut T$ acts on the coordinate functions by
$\tau(x_{p}) = x_{-\tau(p)}$ and $G^{\ast}$ acts by $\sigma(x_{p}) =
{\sigma(\tau_{p})}^{-1}x_{p}$. Taken together this defines an
inclusion $G \oplus G^{\ast} \hookrightarrow
\PGL_{n}(\mathbb{C})$. Now construct a Heisenberg group $H_{T} \simeq
H_{(d,n/d)}$ with Schr\"{o}dinger representation as in
Section~\ref{heis}.

\begin{lemma} \label{inv} A point in $\Def^{a}_{X}$ is $H_{T}$
invariant if and only if it is $G$ invariant.
\end{lemma}
\begin{proof}The induced action of $G^{\ast}$ on $\Def_{X}^{a}$ is
trivial.  In fact the proof of Proposition~\ref{ext} shows that
$G^{\ast}$ is the kernel of ${(\mathbb{C}^{\ast})}^{n-1} \to
\GL(\Def_{X}^{a}(\mathbb{C}[\epsilon]))$.
\end{proof}

Consider the $\mathbb{Z}^{3} \subseteq M^{\prime}$ spanned by
$\epsilon_{1}$, $\epsilon_{2}$ and $\epsilon_{3}$ and set $\bar{M} =
\mathbb{Z}^{3} \cap M$. There is an exact sequence $0 \to \bar{M} \to
M \to K \to 0$. The intersection $\sigma^{\vee} \cap
\bar{M}_{\mathbb{R}}$ is the positive octant $\mathbb{R}^{3}_{\ge 0}$
since
$$
\epsilon_{k} = \frac{1}{n}\sum_{p \in \vertices T} A_{p, \tau_{k}(p)}
\in \sigma^{\vee} \, .$$ Let $\bar{\mathcal{M}} = \Spec \mathbb{C}[\bar{M}
\cap \mathbb{R}^{3}_{\ge 0}]$ be the corresponding $3$-dimensional
toric variety.  If $\bar{G} = \mathbb{Z}^{3}/\bar{M}$ then
$\bar{\mathcal{M}} = \mathbb{C}^{3}/\bar{G}^{\ast}$. We have already
seen $\bar{G}$ in Proposition \ref{ext} and know that as abstract
group it is isomorphic to $G \times \mathbb{Z}_{d}$.

I state the following lemma for lack of reference, the proof is
straightforward.
\begin{lemma} \label{tlemma} Let the cone $\sigma$ and the lattice $N$
determine the affine toric variety $U_{\sigma}$. Assume the
composition of lattice maps $N^{\prime} \hookrightarrow N
\twoheadrightarrow N^{\prime\prime}$ is injective, induces an
isomorphism $\sigma^{\prime} = \sigma \cap N^{\prime}_{\mathbb{R}}
\simeq \sigma^{\prime\prime} = \image \sigma \subset
N^{\prime\prime}_{\mathbb{R}}$ and $\rank N^{\prime} = \rank
N^{\prime\prime}$. If $K = \ker [N \twoheadrightarrow
N^{\prime\prime}]$, then $\ker[T_{N^{\prime}} \to
T_{N^{\prime\prime}}] = T_{N^{\prime}} \cap T_{K}$ and it is the
stabilizer of $U_{\sigma^{\prime}} \subset U_{\sigma}$ in $T_{K}$.
\end{lemma}
\begin{theorem} \label{moduli} The composition
$$
\mathbb{C}[\bar{M} \cap \mathbb{R}^{3}_{\ge 0}] \hookrightarrow \mathbb{C}[M
\cap \sigma^{\vee}] \twoheadrightarrow \mathbb{C}[\mathbb{Z}^{3}_{\ge
0}]$$ where the last map is induced by $p_{1}: M \to \mathbb{Z}^{3}$,
is injective and realizes $\bar{\mathcal{M}}$ as
$\mathcal{M}/\bar{G}^{\ast}$. This identity associates an isomorphism
class of $H_{T}$ invariant smooth 
abelian surfaces to each point of the torus $T_{\bar{M}} \subset
\bar{\mathcal{M}}$.
\end{theorem}
\begin{proof} The injectivity follows from the injectivity of the
composition $\bar{M} \subseteq M \twoheadrightarrow \mathbb{Z}^{3}$.
Dualizing this composition we arrive in the situation of Lemma
\ref{tlemma} with $N^{\prime}$ dual to $\mathbb{Z}^{3}$ and
$N^{\prime\prime}$ dual to $\bar{M}$. From toric geometry it follows
that $\bar{G}^{\ast} \simeq \ker[T_{N^{\prime}} \to
T_{N^{\prime\prime}}]$. Thus by Lemma \ref{tlemma}, $\bar{G}^{\ast}$
is isomorphic to the stabilizer subgroup of $\mathcal{M}$ in $T_{K}$.
The result now follows from the identification $\mathcal{M} =
{(\Def_{X}^{a})}^{G}$, Lemma \ref{tkorb},  Lemma \ref{inv} and the
second statement in Theorem \ref{disc}.
\end{proof}

\begin{remark} Note that $\bar{\mathcal{M}}$ is the normalization of
$\Spec \mathbb{C}[\mathbb{S} \cap \mathbb{Z}^{3}]$. Theorem
\ref{moduli} seems to imply that we see the moduli space for abelian
surfaces with level-structure of type $(d, n/d)$ as an open subset of
one of these spaces. We may at least think of them as representing the
germ at a ``deepest" boundary point. This type of claim presupposes an
analysis of degenerate abelian surfaces arising from the
non-polyhedral equivelar maps on the torus, which is at the moment
work in progress. \end{remark}

With the standard description $\bar{G}$ is the cokernel of the matrix
(\ref{z3}) in the proof of Proposition \ref{ext}.  In each particular
case it is straightforward to describe the action of $\bar{G}$ and
thus the singularities of $\bar{\mathcal{M}}$. Here are two examples.

\begin{example} \label{cycex} Consider the case where $G$ is cyclic
and one of the $\tau_{i}$ generate $G$. We may assume that this is
$\tau_{1}$, so the standard presentation is given by
 $$
\begin{pmatrix} n & b\\ 0 & 1
\end{pmatrix}$$ where $2 \le b \le n-2$. (Not all of these are
polyhedral.) One computes that the dual lattice, $$ \bar{N}
=\mathbb{Z}^{3}+ \mathbb{Z} \cdot \frac{1}{n}\bigl(b(b+1), -b,
b+1\bigr)$$ so the action of $\bar{G}$ is generated by
$\diag(\zeta_{n}^{b(b+1)}, \zeta_{n}^{-b}, \zeta_{n}^{b+1})$ where
$\zeta_{n}$ is a primitive $n$'th root of unity. This will yield an
\emph{isolated} quotient singularity if and only if $n$ is coprime to
both $b$ and $b+1$. This is true if and only if all three $\tau_{k}$
generate $G$.

The quotient singularity will be \emph{Gorenstein} if and only if $1 +
b + b^{2} \equiv 0 \mod n$, which again is equivalent to $T$ being
chiral.  Indeed, if $$ \rho =
\begin{pmatrix} 0 & -1\\ 1 & 1
\end{pmatrix}$$ generates the $6$-fold rotational symmetry in the
standard description, then $T$ is chiral if and only if $\Gamma$, as
translation subgroup of $\Aut(\{3,6\})$, is fixed by conjugacy with
$\rho$. This is again if and only if $\rho (\Gamma) \subset \Gamma$,
when we now view $\Gamma$ as a sublattice of $\mathbb{Z}^{2}$. The
latter is equivalent to
 $$
\begin{pmatrix} n & b\\ 0 & 1
\end{pmatrix}^{-1}
\begin{pmatrix} 0 & -1\\ 1 & 1
\end{pmatrix}
\begin{pmatrix} n & b\\ 0 & 1
\end{pmatrix}$$ being an integral matrix. The condition for this is
exactly $1 + b + b^{2} \equiv 0 \mod n$.
\end{example}

\begin{example} Consider next the case where $G$ is presented by
 $$
\begin{pmatrix} a & 0\\ 0 & c
\end{pmatrix} \, .$$ This is polyhedral if $a,c \ge 3$. In this case
$d = \gcd(a,c)$ and one easily computes that $\bar{M}$ is the image of
 $$
\begin{pmatrix} c & 0 & 0\\ 0 & d & 0\\ 0 & 0 & a
\end{pmatrix} \, .$$ Thus $\bar{\mathcal{M}} = \Spec \mathbb{C}[x^{c},
y^{d}, z^{a}] \simeq \mathbb{C}^{3}$.
\end{example}

\section{The vertex minimal triangulation $T_{7}$} \label{T7} From the
Euler formula $v-e+f=0$ and the fact that $3f=2e$ for surface
manifolds, one concludes that a triangulated torus must have at least
7 vertices. There is exactly one such triangulation and it is
equivelar. It is sometimes called the M\"{o}bius torus, since he gave
the first description in 1861. In 1949 Cs\'{a}z\'{a}r gave the first
polyhedral realization of the triangulation in $3$-space. See
e.g. \cite{be:tou} and the references therein.

\subsection{Invariants and Stanley-Reisner scheme} Call this
triangulation $T_{7}$ - it is drawn in Figure 
\ref{T7_fig}. The group $G$ is $\mathbb{Z}_{7}$ and the standard
presentation is given by
 $$
\begin{pmatrix} 7 & 2\\ 0 & 1
\end{pmatrix}$$ so the pair of divisors is $(1,7)$.  It is chiral and
the automorphism group is the Frobenius group $F_{42} = \mathbb{Z}_{7}
\ltimes \mathbb{Z}_{6}$. The relations among the $\tau_{k}$ are
$\tau_{2} = 4\tau_{1}$ and $\tau_{3} = 2\tau_{1}$.

\begin{figure} \centering
\includegraphics[scale=0.4]{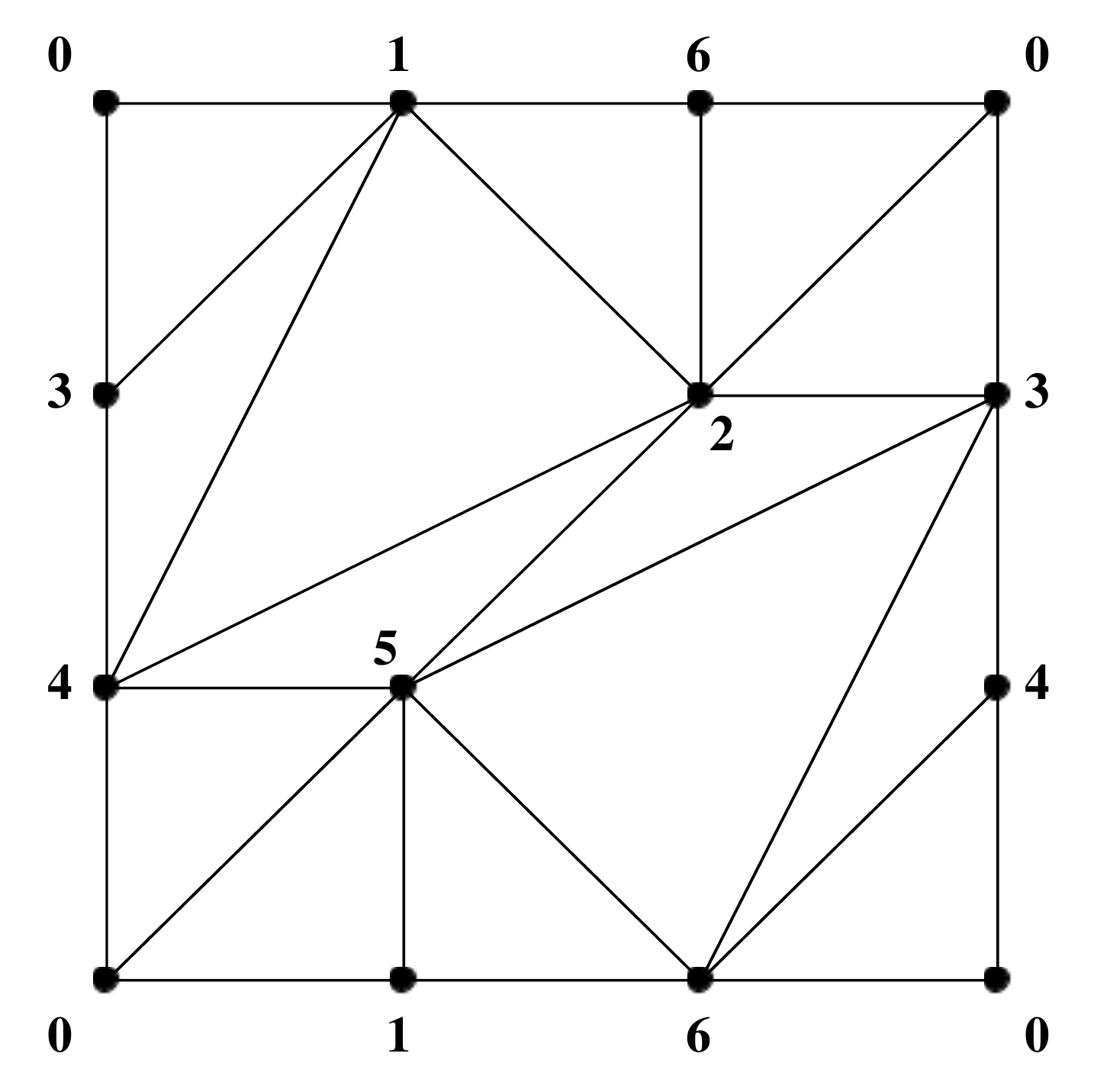}
\caption{The fundmental domain in Figure \ref{lattice} straightens out
  to this standard drawing of the M\"{o}bius torus. The vertices are
  labeled so that $\{m\} = \{(m\tau_{1})(0)\}$.}
\label{T7_fig}
\end{figure}

It is the only polyhedral triangulation with complete edge graph,
i.e. the edge graph is the complete graph $K_{7}$. The Stanley-Reisner
ideal is generated by 21 cubic monomials (see
Section~\ref{SR-scheme}). The Stanley-Reisner scheme consists of $14$
planes.

It appears in the lists of degenerations in both \cite{mr:deg} and
\cite{ma:fam}, in fact Marini calls it the most special
degeneration. It also appears in a central but hidden way in
\cite{ms:mod}. 

I am grateful to the referee for suggesting the
following. In \cite{ms:mod} Manolache and Schreyer prove that the moduli space of
polarized abelian surfaces of type $(1,7)$ is birational to a Fano
threefold commonly known as $V_{22}$. They give an explicit rational
parametrization via a triple projection from a special point $p_e$ in
$V_{22}$.  In their description $p_e$ corresponds to a special cubic
curve in $\mathbb{P}^3$, namely three lines through a point, but by \cite [Proposition 5.16]{ma:fam} this
corresponds again exactly to our Stanley-Reisner scheme. Thus the chart in  \cite{ms:mod}
and our $\bar{\mathcal{M}}$, described later in Section \ref{t7mod}, are in a sense
complimentary.

\subsection{The smoothing component, a reflexive cone and a Calabi-Yau
threefold.} Since $d= 1$, $\Def_{X}^{a}$ has one main component, the unique
smoothing component $S$.

\begin{theorem} The cone $\sigma^{\vee}$ for $T_{7}$ is a $9$
dimensional completely split reflexive Gorenstein cone of index
$3$. It is the Cayley cone over three $6$-dimensional lattice
simplices.
\end{theorem}

\begin{proof} The result follows from implementing our general results
from Section~\ref{conesec}. I have done the computations in Maple
using the package {\tt Convex} (\cite{fr:co}). Choose an origin in
$\vertices T$.  Since $T_{7}$ is chiral, $\bar{M}$ is Gorenstein, so
$\epsilon_{1} + \epsilon_{2} + \epsilon_{3} \in M$. In fact $$
\epsilon_{1} + \epsilon_{2} + \epsilon_{3} = A_{p,\tau_{1}(p)}
+A_{\tau_{3}(p),(\tau_{2}+\tau_{3})(p)} +
A_{-\tau_{2}(p),(\tau_{3}-\tau_{2})(p)} \, .$$ As a basis for $M$
choose
\begin{gather*}E_{1} = A_{0,\tau_{1}(0)}, E_{2} =
A_{\tau_{3}(0),(\tau_{2}+\tau_{3})(0)}, E_{3}=
A_{-\tau_{2}(0),(\tau_{3}-\tau_{2})(0)}\\ E_{3+i} =
m_{(i-1)\tau_{1}(0)}
\end{gather*} for $i = 1, \dots , 6$, where the $m_p$ are as in
\eqref{them}. Thus in this basis $E_{i}^{\ast}=  \epsilon_{i}^{\ast}$ for $i=1,2,3$.  

After expressing the columns of $A$ in this basis, i.e. after turning
$M$ into the standard $\mathbb{Z}^{9}$, one computes that
$\sigma^{\vee}$ is the Cayley cone over the 3 simplices in
$\mathbb{R}^{6}$ described in Table \ref{simpl}. Now plug this into
the computer program and find that $\sigma$, the dual cone, has 24
rays, three of them are of course spanned by $\epsilon_{1}^{\ast},
\epsilon_{2}^{\ast}, \epsilon_{3}^{\ast}$.  Moreover
$m_{\sigma}:=\epsilon_{1} + \epsilon_{2} + \epsilon_{3} \in M$ yields
the Gorenstein property on $\sigma$.  This proves the theorem.
\end{proof}

\begin{table} \centering $$ \Delta_{1}: \left[
\begin {smallmatrix} 0&0&0&0&0&0&-1\\ 0&1&1&1&1&1&0\\ 0&0&1&1&1&1&0\\
0&0&0&1& 1&1&0\\ 0&0&0&0&1&1&0\\ 0&0&0&0&0&1&0
\end {smallmatrix} \right] \quad \Delta_{2}: \left[
\begin {smallmatrix} 0&0&0&-1&0&0&-1\\ -1&0 &0&-1&-1&0&-1\\
-1&-1&0&-1&-1&-1&-1 \\ 0&0&0&0&0&0&-1\\ -1&0&0&-1&0&0&-1 \\
-1&-1&0&-1&-1&0&-1
\end {smallmatrix} \right] \quad \Delta_{3}: \left[
\begin {smallmatrix} 1&0&1&0&1&0&0\\ 0&0&0 &0&0&0&-1\\ 0&0&1&0&1&0&0\\
0&-1&0&0 &0&0&-1\\ 0&0&0&0&1&0&0\\ 1&0&1&0&1& 1&0
\end {smallmatrix} \right] $$
\caption{The simplices are the convex hull of the columns.}
\label{simpl}
\end{table}

\begin{corollary} The smoothing component $S$ is the germ of a normal
Gorenstein affine toric variety.
\end{corollary}
\begin{proof} Since $\sigma$ is Gorenstein, the point
$m_{\sigma}:=\epsilon_{1} + \epsilon_{2} + \epsilon_{3} \in M$
generates the interior of $\sigma^{\vee}$, in the sense that
$\interior \sigma \cap M = m_{\sigma} + \sigma \cap M$. But in this
example $m_{\sigma}$ is in the semigroup generated by the columns of
$A$, so the semigroup ring is normal.
\end{proof}

The notion of reflexive cones was introduced by Batyrev and Borisov to
study mirror symmetry for complete intersection Calabi Yau manifolds
in toric varieties. In our example, since $\sigma^{\vee}$ is
reflexive, the Minkowski sum $\Delta = \Delta_{1} + \Delta_{2}
+\Delta_{3}$ will be a \emph{reflexive} 6 dimensional polytope. This
polytope determines a 6 dimensional (singular) Fano variety $Y$. The
Minkowski decomposition gives us 3 divisors $E_{i}$ on $Y$. If we cut
$Y$ with a general section of each of the $\mathcal{O}_{Y}(E_{i})$ the
result is a singular Calabi-Yau 3-fold and one can now take a crepant
resolution to arrive at a Calabi-Yau 3-manifold.

Maxmillian Kreuzer ran the polytope through the computer program PALP
(see \cite{ks:pa}) with the following result.
\begin{proposition} The Calabi Yau 3-manifold arising from the
reflexive cone for $T_{7}$ has Hodge numbers $h^{1,1} = 15$ and
$h^{1,2} = 12$. In particular the Euler number is 6.
\end{proposition}

More may be said about $\Delta$ but I have not been able to find a
good description of it. The $f$-vector is $(112,427,630,441,147,21,1)$
and the facets are probably all isomorphic. They alle have at least the same
$f$-vector $(38,111,125, 64, 14, 1)$. From the PALP computation one
learns that it has $204 = 29 \times 7 +1$ lattice points. It might be
easier to describe the dual reflexive polytope since it has $21$
vertices and $22$ (!) lattice points.

Note that if $Z$ is the total space of the vector bundle
$\bigoplus_{i=1}^{3}\mathcal{O}_{Y}(E_{i})$, toric geometry tells us
that there is a birational toric morphism $f: (Z,Y) \to (S,0)$, where
we identify $Y$ as the zero section.

\begin{remark} Since the cones associated to all equivelar
triangulated tori are Cayley cones, one could ask if more of them are
reflexive. I have not been able to prove, but do conjecture that
$T_{7}$ is the \emph{only} polyhedral triangulation leading to a
reflexive cone. There are other non-polyhedral examples.
\end{remark}

\subsection{The non-smoothing components and tilings of the torus} The
versal base space in this case is defined by $21$ binomials in 
$21$ variables. The numbers are small enough for us to be able to give
the full component structure. This was done by delicate use of the
ideal quotient command in Macaulay 2 (\cite{m2}).
\begin{proposition} The versal base space
$\Def_{\mathbb{P}(T_{7})}^{a}$ is reduced. It is the union of $29$
irreducible components. The $28$ non-smoothing components are
isomorphic to the germ of the product of the affine cone over
$\mathbb{P}^{1} \times \mathbb{P}^{2} \subset \mathbb{P}^{5}$ and
$\mathbb{C}^{3}$ or $\mathbb{C}^{4}$.
\end{proposition}

When one has the ideal of the component it is in this case easy to
find the generic fiber. This will be an interesting scheme since it is
a ``generic" non-smoothable degenerate abelian surface, i.e. it cannot
appear in degenerations of smooth abelian surfaces.

\begin{figure}
  \centering
  \subfloat[$P_1$]{\includegraphics[width=0.5\textwidth]{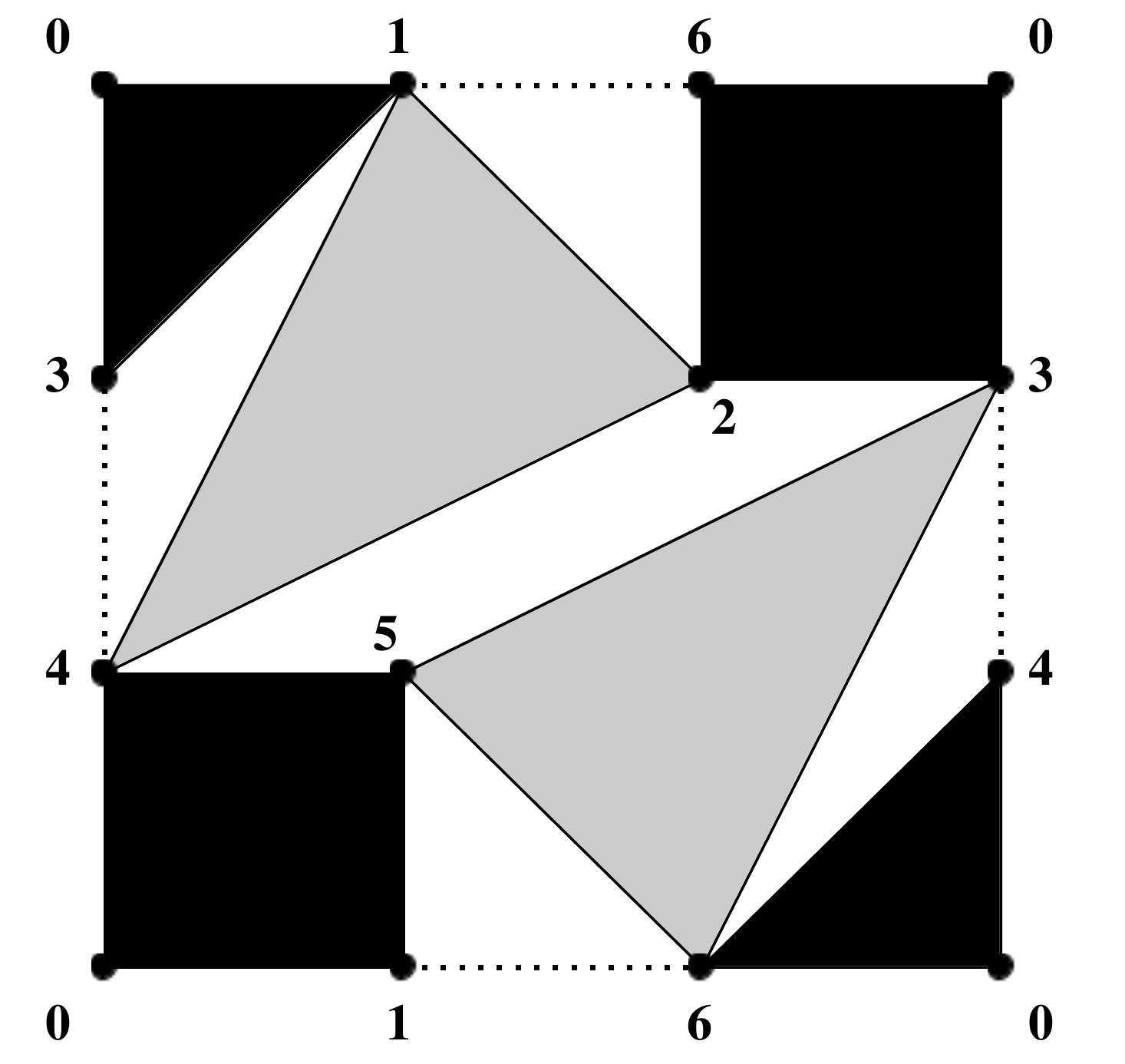}}                
  \subfloat[$P_2$]{\includegraphics[scale=0.523]{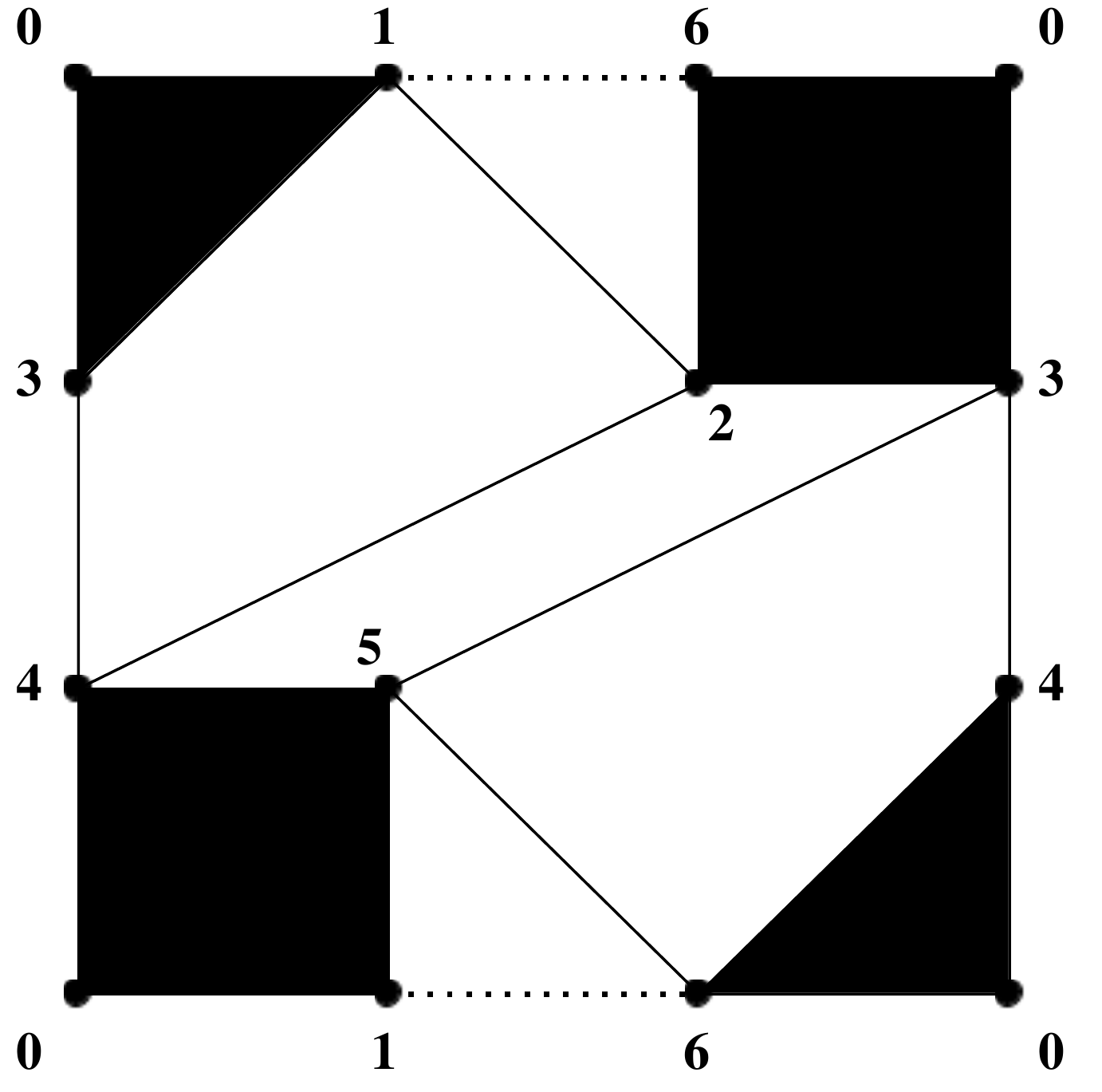}}
\caption{The two rigid tilings of the torus.}
  \label{rigidP}
\end{figure}

The 28 components come in $4$ $G$ orbits, but 3 of these (the 8
dimensional ones) have isomorphic generic fibers. In Figure
\ref{rigidP} I have drawn two tilings of the torus, $P_{1}$ with 1
hexagon (in black), 3 quadrangles and 2 triangles and $P_{2}$ with 1
hexagon and 4 quadrangles. We can associate an embedded rational
projective surface to each polygon. A hexagon corresponds to a Del
Pezzo surface of degree $6$ in $\mathbb{P}^{6}$, a quadrangle
corresponds to $\mathbb{P}^{1} \times \mathbb{P}^{1}$ embedded via the
Segre embedding in $\mathbb{P}^{3}$ and a triangle corresponds to
$\mathbb{P}^{2}$. Now take the union of them in $\mathbb{P}^{6}$,
intersecting as in $P$, to make the degenerate abelian surface
$X_{P}$.
\begin{proposition} The generic fiber over a component of dimension
$7$ is $X_{P_{1}}$ and over a component of dimension $8$ it is
$X_{P_{2}}$.
\end{proposition}
\begin{remark} It is probably better to think of the tilings above as
periodic polygonal tilings of the plane with vertices contained in the
lattice of vertices of $\{3, 6\}$, see Figure \ref{rigidT}. Note also that $\Gamma$ is the full
translation group of the tiling. All such tilings can be constructed
by erasing $\Gamma$ orbits of edges in $\{3, 6\}$. I believe that in
general the generic fiber over a non-smoothing component may be
described by erasing the edges corresponding to deformation parameters
that do \emph{not} vanish on the whole component. It is tempting to
conjecture that the components of the non-smoothable fiber are the
projective toric varieties associated to the $\{3,6\}$ lattice
polygons in the tiling.
\end{remark}

\begin{figure}
  \centering
  \subfloat{\includegraphics[width=0.45\textwidth]{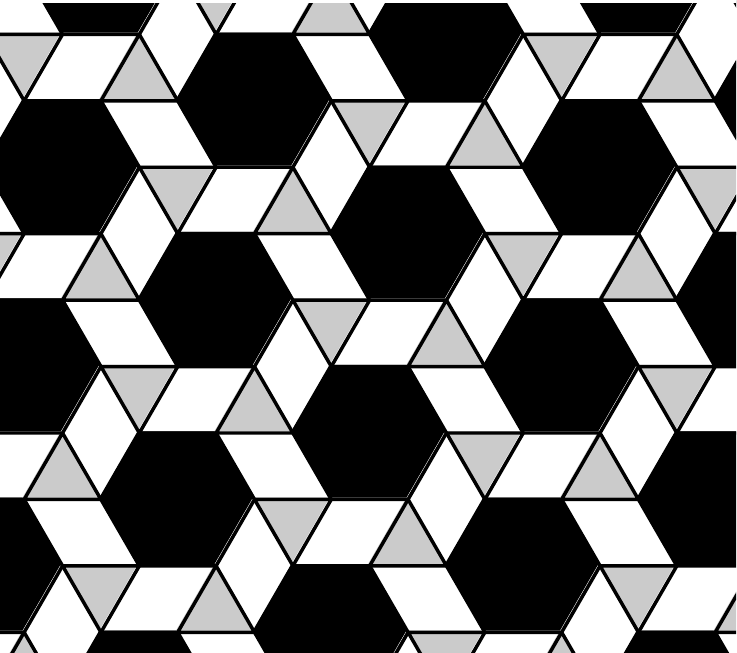}}   \qquad             
  \subfloat{\includegraphics[width=0.45\textwidth]{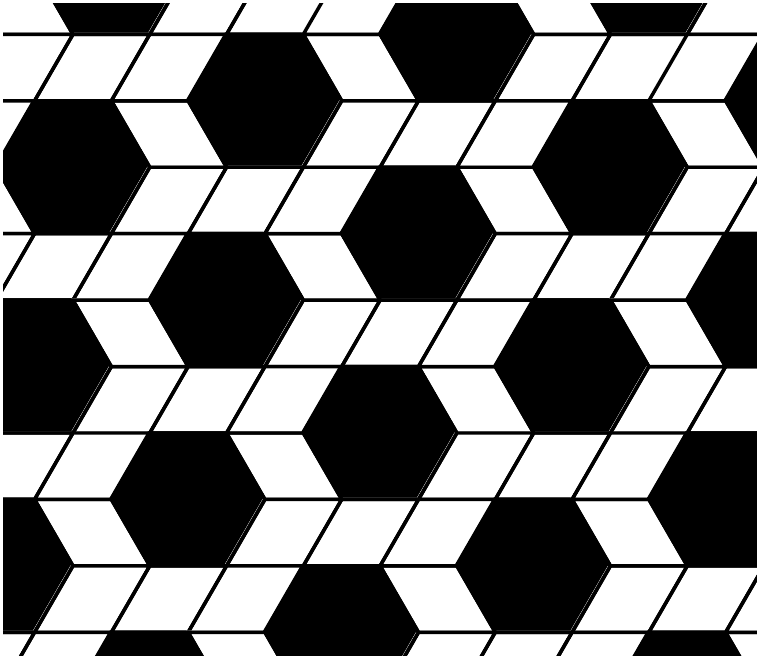}}
\caption{Tilings of the plane that cover the tilings in Figure \ref{rigidP}.}
  \label{rigidT}
\end{figure}

\subsection{The space $\bar{\mathcal{M}}$.} \label{t7mod}
Since the smoothing component $S$ is normal in this case the
$\bar{\mathcal{M}}$ of Theorem \ref{moduli} equals
$\Spec \mathbb{C}[\mathbb{S} \cap \mathbb{Z}^{3}]$. From Example
\ref{cycex} we see that $\bar{\mathcal{M}} \simeq
\mathbb{C}^3/\mathbb{Z}_7$ where $\mathbb{Z}_7$ acts with generator
$\diag(\zeta_{7}^{6}, \zeta_{7}^{5}, \zeta_{7}^{3}) \in \SL_{3}(\mathbb{C})$. One may check
using the criteria of Reid, Shephard-Barron and Tai (see
e.g. \cite{re:can}) that the singularity at the origin is canonical.

\subsection{The equations of a Heisenberg invariant smooth surface} Using this
deformation theory we can find equations for Heisenberg  
invariant abelian surfaces in $\mathbb{P}^{6}$. By Lemma \ref{tktm} we
need to find the family over the smooth subspace $\mathcal{M}$ with
three parameters $u_{k}= t_{p, \tau_{k}(p)}$.

Let us first index the vertices by their $\tau_{1}$ orbit, i.e. after
fixing a vertex $\{0\}$, $\{m\} = \{m\tau_{1}(0)\}$. Thus in cycle
notation $\tau_{1} = (0,1,2,3,4,5,6)$. Then $\tau_{2}$ becomes $(0, 4,
1, 5, 2, 6, 3 )$ and $\tau_{3}$ is $(0, 2, 4, 6, 1, 3, 5)$. The group
$\Aut(T) = F_{42}$ is generated by $\tau_{1}$ and the rotation $\rho =
(1,5,4,6,2,3)$. Note that $\rho$ acts on $\mathcal{M}$ as the
permutation $(u_{1}, u_{3}, u_{2})$.

The 21 cubic monomials generating $I_{X}$ are one $F_{42}$ orbit, but
it is convenient to partition them in 3 $\tau_{1}$ orbits since they
have the nice form $$ x_{-\tau_{k}(p)}x_{p}x_{\tau_{k}(p)} \quad
k=1,2,3 \quad p \in \vertices T \, .$$ Note again that $\rho$ permutes
these 3 orbits.

The first order deformations are easily found and the first order
family is defined by the $\tau_{1}$ orbits of $$ x_{0}x_{1}x_{6} +
u_{2}(x_{1}^{2}x_{5}+ x_{2}x_{6}^{2}), x_{0}x_{3}x_{4} +
u_{3}(x_{4}^{2}x_{6}+ x_{1}x_{3}^{2}), x_{0}x_{2}x_{5} +
u_{1}(x_{2}^{2}x_{3}+ x_{4}x_{5}^{2}) \, .$$ Instead of lifting
equations and relations to continue apply the symmetry.

The ideal must be Heisenberg invariant. In our case the action of
$G^{\ast}$ is generated by $x_{m} \mapsto \zeta_{7}^{m}x_{m}$. Thus if
$x_{m_{1}}x_{m_{2}}x_{m_{3}}$ is a term in the perturbation of
$x_{-\tau_{k}(m)}x_{m}x_{\tau_{k}(m)}$ we must have $$ m_{1} + m_{2} +
m_{3} \equiv -\tau_{k}(m)+ m + \tau_{k}(m) \mod 7 \, .$$ Moreover
$\rho^{3} = (1,6)(2,5)(3,4)$ fixes each
$x_{-\tau_{k}(m)}x_{m}x_{\tau_{k}(m)}$, so terms in $\rho^{3}$ orbits
in the perturbation must have the same coefficient. Finally we may
exclude terms that are in $I_{X}$. The upshot is that the family is
defined by the orbits of
\begin{equation}
\begin{split} \label{fam} x_{0}x_{1}x_{6} &+ u_{2}(x_{1}^{2}x_{5}+
x_{2}x_{6}^{2}) + \psi_{1}(x_{1}x_{2}x_{4}+x_{3}x_{5}x_{6}) \\ &+
\varphi_{1}x_{0}^{3}+\xi_{1}(x_{1}x_{3}^{2}+ x_{4}^{2}x_{6}) +
\upsilon_{1}(x_{2}^{2}x_{3}+ x_{4}x_{5}^{2})\\ x_{0}x_{3}x_{4} &+
u_{3}(x_{4}^{2}x_{6}+ x_{1}x_{3}^{2}) +
\psi_{2}(x_{1}x_{2}x_{4}+x_{3}x_{5}x_{6}) \\ &+ \varphi_{2}x_{0}^{3} +
\xi_{2}(x_{4}x_{5}^{2}+ x_{2}^{2}x_{3}) + \upsilon_{2}(x_{1}^{2}x_{5}+
x_{2}x_{6}^{2})\\ x_{0}x_{2}x_{5} &+ u_{1}(x_{2}^{2}x_{3}+
x_{4}x_{5}^{2}) + \psi_{3}(x_{1}x_{2}x_{4}+x_{3}x_{5}x_{6}) \\ &+
\varphi_{3}x_{0}^{3} + \xi_{3}(x_{2}x_{6}^{2}+ x_{1}^{2}x_{5}) +
\upsilon_{3}(x_{4}^{2}x_{6}+ x_{1}x_{3}^{2})
\end{split}
\end{equation} where the $\xi_{i}, \upsilon_{i}, \varphi_{i},
\psi_{i}$ are power series in $u_{1}, u_{2}, u_{3}$. Thus the task
becomes to find similar expressions for lifted relations and then
solving functional equations to make the family flat.

I describe here the answer only for the one parameter deformation $s =
u_{1} = u_{2} = u_{3}$. In this case also $\psi_{1} = \psi_{2} =
\psi_{3}$ etc.\ so denote the common function by $\psi$. Let $f(s)$ be
a power series solution for the equation $$
{s}^{6}{f(s)}^{4}-{s}^{4}(s+1){f(s)}^{3}-(s+1)^{2}(s-1)f(s)+(s +1)^{2}
= 0\, .$$
\begin{proposition} The family defined by the orbits of equations
(\ref{fam}) form a flat one parameter smoothing if $s = u_{1} = u_{2}
= u_{3}$ and
$$
\varphi = -\psi = \frac{s^{2}(s^{4}f(s)^{3}-s-1)}{1+s}, \quad \xi =
s^{2}f(s), \quad \upsilon = \frac{s^{4}f(s)^{2}}{1+s} \, .$$
\end{proposition} Because of the symmetry only 2 relations need to be
lifted. I computed the liftings using Maple.

 \bibliographystyle{amsalpha}

\begin{thebibliography}{DMM10}

\bibitem[AC10]{ac:def}
Klaus Altmann and Jan~Arthur Christophersen, \emph{Deforming
  {S}tanley-{R}eisner schemes}, Math. Ann. \textbf{348} (2010), 513--537.

\bibitem[Art69]{ar:alg}
M.~Artin, \emph{Algebraic approximation of structures over complete local
  rings}, Inst. Hautes \'Etudes Sci. Publ. Math. (1969), no.~36, 23--58.

\bibitem[BB97]{bb:dua}
Victor~V. Batyrev and Lev~A. Borisov, \emph{Dual cones and mirror symmetry for
  generalized {C}alabi-{Y}au manifolds}, Mirror symmetry, II, AMS/IP Stud. Adv.
  Math., vol.~1, Amer. Math. Soc., 1997, pp.~71--86.

\bibitem[BE91a]{be:gra}
Dave Bayer and David Eisenbud, \emph{Graph curves}, Adv. Math. \textbf{86}
  (1991), no.~1, 1--40.

\bibitem[BE91b]{be:tou}
J{\"u}rgen Bokowski and Anselm Eggert, \emph{Toutes les r\'ealisations du tore
  de {M}\"obius avec sept sommets}, Structural Topology \textbf{17} (1991),
  59--78.

\bibitem[BK08]{bk:equ}
Ulrich Brehm and Wolfgang K{\"{u}}hnel, \emph{Equivelar maps on the torus},
  Eur. J. Comb. \textbf{29} (2008), no.~8, 1843--1861.

\bibitem[BN08]{bn:com}
Victor Batyrev and Benjamin Nill, \emph{Combinatorial aspects of mirror
  symmetry}, Integer points in polyhedra (Matthias~Beck et~al., ed.),
  Contemporary Mathematics, vol. 452, AMS, 2008, pp.~35--66.

\bibitem[Cox49]{co:con}
H.~S.~M. Coxeter, \emph{Configurations and maps}, Rep. Math. Colloquium (2)
  \textbf{8} (1949), 18--38.

\bibitem[DMM10]{dmm:com}
Alicia Dickenstein, Laura~Felicia Matusevich, and Ezra Miller,
  \emph{Combinatorics of binomial primary decomposition}, Math. Z. \textbf{264}
  (2010), 745--763.

\bibitem[ES96]{es:bin}
David Eisenbud and Bernd Sturmfels, \emph{Binomial ideals}, Duke Math. J.
  \textbf{84} (1996), 1--45.

\bibitem[Fra09]{fr:co}
Matthias Franz, \emph{Convex - a maple package for convex geometry},
  http://www.math.uwo.ca/{~}mfranz/convex/, 2009, Version 1.1.3.

\bibitem[Ful93]{fu:int}
William Fulton, \emph{Introduction to toric varieties}, Annals of Mathematics
  Studies, vol. 131, Princeton University Press, 1993.

\bibitem[GP98]{gp:equ}
Mark Gross and Sorin Popescu, \emph{Equations of $(1, d)$-polarized abelian
  surfaces}, Math. Ann. \textbf{310} (1998), 333--377.

\bibitem[GP01]{gp:cal}
\bysame, \emph{Calabi-{Y}au threefolds and moduli of abelian surfaces. {I}},
  Compositio Math. \textbf{127} (2001), no.~2, 169--228.

\bibitem[GS]{m2}
Daniel~R. Grayson and Michael~E. Stillman, \emph{Macaulay2, a software system
  for research in algebraic geometry}, http://www.math.uiuc.edu/Macaulay2/.

\bibitem[HS94]{hs:kod}
K.~Hulek and G.~K. Sankaran, \emph{The {K}odaira dimension of certain moduli
  spaces of abelian surfaces}, Compositio Math. \textbf{90} (1994), 1--35.

\bibitem[KS04]{ks:pa}
Maximilian Kreuzer and Harald Skarke, \emph{Palp: A package for analysing
  lattice polytopes with applications to toric geometry}, Comput. Phys. Commun.
  \textbf{157} (2004), no.~1, 87--106.

\bibitem[Mar04]{ma:fam}
Alfio Marini, \emph{On a family of {$(1,7)$}-polarised abelian surfaces}, Math.
  Scand. \textbf{95} (2004), 181--225.

\bibitem[MR05]{mr:deg}
F.~Melliez and K.~Ranestad, \emph{Degenerations of {$(1,7)$}-polarized abelian
  surfaces}, Math. Scand. \textbf{97} (2005), 161--187.

\bibitem[MS01]{ms:mod}
Nicolae Manolache and Frank-Olaf Schreyer, \emph{Moduli of {$(1,7)$}-polarized
  abelian surfaces via syzygies}, Math. Nachr. \textbf{226} (2001), 177--203.

\bibitem[Mum66]{mu:one}
D.~Mumford, \emph{On the equations defining abelian varieties. {I}}, Invent.
  math. \textbf{1} (1966), 287--354.

\bibitem[Neg83]{neg:uni}
Seiya Negami, \emph{Uniqueness and faithfulness of embedding of toroidal
  graphs}, Discrete Math. \textbf{44} (1983), 161--180.

\bibitem[PT10]{pt:tor}
Pedro Daniel~Gonzalez Perez and Bernard Teissier, \emph{Toric geometry and the
  {S}emple-{N}ash modification}, arXiv:0912.0593v2 [math.AG], 2010.

\bibitem[Rei80]{re:can}
Miles Reid, \emph{Canonical {$3$}-folds}, Journ\'ees de {G}\'eometrie
  {A}lg\'ebrique d'{A}ngers, {J}uillet 1979/{A}lgebraic {G}eometry, {A}ngers,
  1979, Sijthoff \& Noordhoff, 1980, pp.~273--310.

\bibitem[Ser06]{ser:def}
Edoardo Sernesi, \emph{Deformations of algebraic schemes}, Springer-Verlag,
  2006.

\bibitem[Sta96]{sta:com}
Richard~P. Stanley, \emph{Combinatorics and commutative algebra}, second ed.,
  Birkh{\"{a}}user Boston Inc., 1996.

\end{thebibliography}

\providecommand{\bysame}{\leavevmode\hbox to3em{\hrulefill}\thinspace}
\providecommand{\MR}{\relax\ifhmode\unskip\space\fi MR }
\providecommand{\MRhref}[2]{%
  \href{http://www.ams.org/mathscinet-getitem?mr=#1}{#2}
}
\providecommand{\href}[2]{#2}

\end{document}